\documentclass[dvips,12pt]{amsart} \font\emailfont=cmtt10
\pdfoutput=0

\usepackage{hyperref,amsmath,amsthm,amsfonts,amscd,flafter,epsf,verbatim,cite,graphicx,color}
\usepackage[all]{xy}

\newtheorem{theorem}{Theorem}[section]
\newtheorem{lemma}[theorem]{Lemma}
\newtheorem{proposition}[theorem]{Proposition}
\newtheorem{corollary}[theorem]{Corollary}

\theoremstyle{definition} \newtheorem{definition}[theorem]{Definition}

\theoremstyle{remark} \newtheorem{remark}[theorem]{Remark}

\newcommand{\Q}{\mathbb{Q}} \newcommand{\Z}{\mathbb{Z}}
\newcommand{\N}{\mathbb{N}}

\title[{The rational Khovanov homology of $3$-strand pretzel links}]
{The rational Khovanov homology of $3$-strand pretzel links}

\author[Andrew Manion]{Andrew Manion} \address{Department of
  Mathematics, Princeton University, New Jersey 08544 \newline
  \indent{\emailfont{amanion@math.princeton.edu}}} \thanks{The author
  was supported by the Department of Defense (DoD) through the
  National Defense Science and Engineering Graduate Fellowship (NDSEG)
  Program.}

\begin{document}

\begin{abstract}
  The $3$-strand pretzel knots and links are a well-studied source of
  examples in knot theory. However, while there have been computations
  of the Khovanov homology of some sub-families of $3$-strand pretzel
  knots, no general formula has been given for all of them. We give a
  general formula for the unreduced Khovanov homology of all
  $3$-strand pretzel links, over the rational numbers.
\end{abstract}

\maketitle
\section{Introduction}

The goal of this paper is to compute the unreduced Khovanov homology,
over $\Q$, of all $3$-strand pretzel links. The literature contains
computations for some pretzel knots; in \cite{Suzuki}, Suzuki computes
the Khovanov homology of $(p,2-p,-r)$ pretzel knots with $p \geq 9$
odd and $r \geq 2$ even. More generally, the knots considered in
\cite{Suzuki} are quasi-alternating (see \cite{CK} and \cite{Greene}),
and for quasi-alternating links the Khovanov homology can be computed
directly from the Jones polynomial and signature (see \cite{MO}). In
\cite{Starkston}, Starkston considers a family of
non-quasi-alternating knots, the $(-p,p,q)$ pretzel knots for odd $p$
and $q \geq p$. She computes the Khovanov homology of many of
them. However, the Khovanov homology of most non-quasi-alternating
pretzel knots has not appeared in the literature. We will complete
this computation for all (non-quasi-alternating) pretzel links, over
$\Q$.

As one may expect, we use the unoriented skein exact sequence in
Khovanov homology; see \cite{KPKH} for a discussion of this
sequence. We will use the sequence in an inductive argument,
unraveling strands of a pretzel link one crossing at a time. In order
to deal with all pretzel links, one must structure the induction
carefully, as we will see. One must also use the Lee spectral sequence
(see \cite{Lee}) to help with many cases of the inductive step.

Before beginning, we will briefly review our conventions on $3$-strand
pretzel links. The $(l,m,n)$ pretzel link will be denoted
$P(l,m,n)$. (The standard letters to use are $p$, $q$, and $r$, but we
have chosen the letters $l$, $m$, and $n$ instead since $q$ will be
used for the quantum grading on Khovanov homology.) The knot
$P(-3,5,7)$ is shown in Figure~\ref{357}. It should remind the reader
of the general form. For the link $P(l,m,n)$, positive values of $l$,
$m$, and $n$ represent right-handed strands, and negative values
represent left-handed strands. Note that $P(l,m,n)$ is a knot when
zero or one of $\{l,m,n\}$ are even, a two-component link when two of
$\{l,m,n\}$ are even and one is odd, and a three-component link
otherwise. The ordering of $l$, $m$, and $n$ does not matter; it is
clear that cyclic permutations of the three strands do not change the
link, and transpositions simply amount to turning the link on its
head.

\begin{figure}
  \begin{center}
    \epsfbox{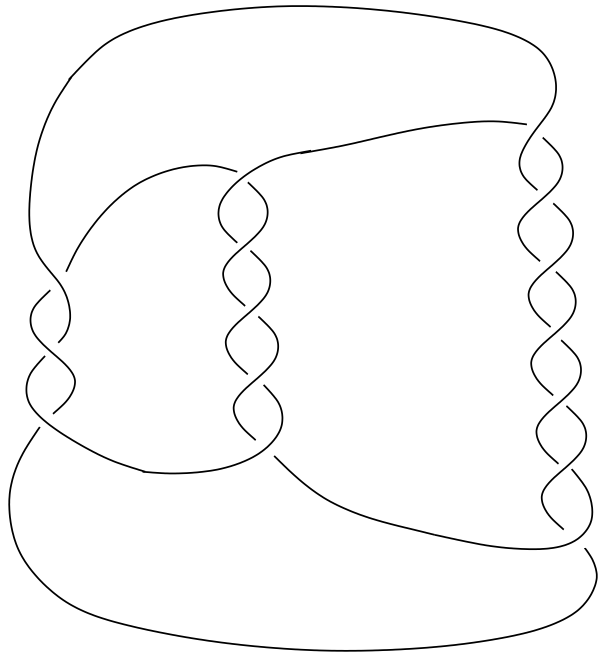}
  \end{center}
  \caption{\label{357} The $(-3,5,7)$ pretzel knot.}
\end{figure}

For quasi-alternating links $L$, the Khovanov homology was computed to
be thin in \cite{MO}. It depends only on the Jones polynomial and
signature of $L$. If $l$, $m$, and $n$ are all positive, then
$P(l,m,n)$ is actually alternating, so its Khovanov homology is known
(the Jones polynomials of pretzel links can be found in
\cite{Landvoy}). Also, mirroring the knot $P(l,m,n)$ gives
$P(-l,-m,-n)$. Since Khovanov homology enjoys a symmetry under
mirroring (the quantum and homological gradings are replaced by their
negatives), we only need consider $P(-l,m,n)$ when $l$, $m$, and $n$
are positive integers.

Champanerkar and Kofman~\cite{CK} determined the quasi-alternating
status of most pretzel links, and Greene~\cite{Greene} finished the
rest. The following special case of Greene's results will be all we
need:

\begin{theorem}[Greene \cite{Greene}]\label{qa}
  $P(-l,m,n)$ is quasi-alternating if and only if $l > \min\{m,n\}$.
\end{theorem}

Hence, to consider all non-quasi-alternating preztel links, it will
suffice to consider $P(-l,m,n)$ with $2 \leq l \leq m \leq n$.

When we refer to Khovanov homology in this paper, we will always be
using coefficients in $\Q$. The Khovanov homology of a link takes the
form of a bigraded vector space over $\Q$. The two gradings will be
denoted $q$, the quantum grading, and $t$, the homological
grading. The bigraded vector space $Kh(P(-l,m,n))$ will be of the form
$L \oplus U$, where $L$ and $U$ are spaces to be defined in
Section~\ref{genform}. More precisely, we have
\begin{theorem}\label{intromaintheorem} Suppose $2 \leq l \leq m \leq
  n$. Then
  \[
  Kh(P(-l,m,n)) = q^{\sigma_L} t^{\tau_L} L_{l,m,n} \oplus
  q^{\sigma_U} t^{\tau_U} U_{l,m,n},
  \]
  where $L_{l,m,n}$ is a bigraded vector space specified in
  Definition~\ref{lower}, $U_{l,m,n}$ is a bigraded vector space
  specified in Definition~\ref{upper}, and the values of $\sigma_L$,
  $\tau_L$, $\sigma_U$, and $\tau_U$ are specified in
  Proposition~\ref{GradingShiftProp}. (The multiplications by
  monomials in $q$ and $t$ simply shift the bigradings by the
  indicated amount).
\end{theorem}
The $L$ summand will depend mostly on $l$ and not on $m$ or $n$, and
the $U$ summand will depend mostly on $m - l$ and $n - m$ (``mostly''
means there are still some cases depending on parity and on whether $m
= l$).

We will briefly sketch here how $L$ and $U$ are defined; the details
are in Section~\ref{genform}. First, the spaces involved in the
formula for $Kh(P(-l,m,n))$ will all be contained in three adjacent
$\delta$-gradings. In fact, all pretzel knots with arbitrarily many
strands have the same property, as Champanerkar and Kofman point out
in \cite{CK}: pretzel knots with arbitrarily many strands have Turaev
genus $1$ (there is a nice picture-proof in \cite{CK}), and unreduced
homological width is bounded by the Turaev genus plus $2$ (see
\cite{DFKLS}, \cite{CKS}).

Plot $Kh(P(-l,m,n)) = L \oplus U$ on a two-dimensional grid with the
homological grading $t$ on the $x$-axis and the quantum grading $q$ on
the $y$-axis (see Figure~\ref{Kh357}). The summands $L$ and $U$ will
mostly occupy different regions, or regions which only slightly
overlap, and $L$ is to the left of $U$ and below it. (Thus the $q$-
and $t$-gradings of generators of $L$ are generally lower than those
of $U$, explaining the use of the letters $L$ for ``lower'' and $U$
for ``upper.'') In the generic case $m \neq l$, the summands $L$ and
$U$ share no columns of the grid when $l$ is odd and two columns when
$l$ is even.

Each summand $L$ and $U$ will be contained in only two of the three
possible $\delta$-gradings. $L$ will be contained in the higher two
$\delta$-gradings, and $U$ will be contained in the lower
two. Furthermore, $L$ and $U$ will each be made up of knight's moves
and exceptional pairs (see Figure~\ref{knightsep}). Hence, we only
need to specify the values of $L$ or $U$ in one of its two
$\delta$-gradings, as long as we know where its exceptional pairs
are. The required data are a sequence of integers, representing
dimensions of the summand $L$ or $U$ in its bottom $\delta$-grading,
together with the bigrading of one generator of the summand (to fix
the overall gradings) and the locations of the exceptional pairs. In
Section~\ref{genform}, we will define $L$ and $U$ by specifying these
data.

\subsection{Acknowledgements} The author would like to thank
Zolt{\'a}n Szab{\'o} for helpful conversations during the writing of
this paper.

\section{The general formula}\label{genform}

\subsection{Bigraded vector spaces.}
Let $V$ be a bigraded vector space. We will write $V = \oplus_{i,j}
V_{i,j}$, where $V_{i,j}$ denotes the subspace in $t$-grading $i$ and
$q$-grading $j$. The space $V$ can then be specified by its
Poincar{\'e} polynomial, a Laurent polynomial $P_V$ in $q$ and $t$
with positive integer coefficients such that the coefficient of $q^j
t^i$ in $P_V$ is the dimension of $V_{i,j}$ (note the ordering of the
indices). We will henceforth identify $V$ with $P_V$.

Monomial multiplication on $P_V$ corresponds to grading-shift on $V$:
if $a$ and $b$ are integers, then $q^a t^b V$ is defined as the
bigraded vector space with Poincar{\'e} polynomial $q^a t^b
P_V$. Equivalently,
\[
(q^a t^b V)_{i,j} = V_{i-b,j-a}.
\]

Visually, one depicts a bigraded vector space by drawing a grid (see
Figure~\ref{Kh357}). Our convention will be to put the $t$-grading on
the horizontal axis and the $q$-grading on the vertical axis
(labelling only odd values or only even values of $q$, since the
vector spaces in question will always have only odd or only even
$q$-gradings).

The $\delta$-grading on a bigraded vector space $V$ is defined by
$V_{\delta} = \oplus_{j-2i=\delta} V_{i,j}$ (see \cite{KPKH}). It
corresponds to summing $V$ along diagonals.

\subsection{An example, and some definitions.}

\begin{figure}
  \begin{center}
    \input{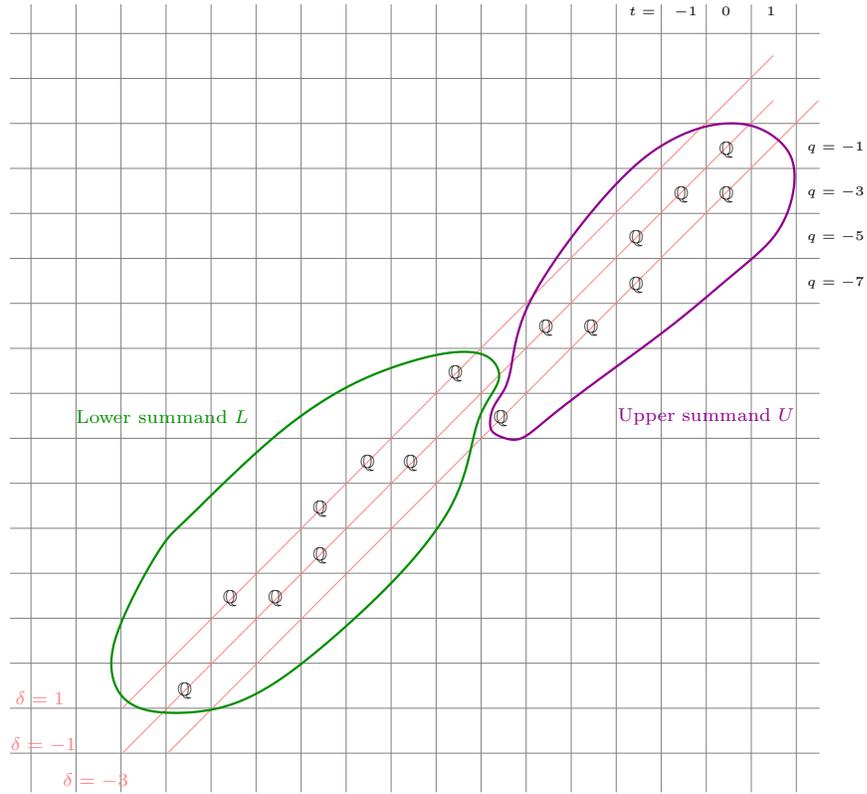}
  \end{center}
  \caption{\label{Kh357} Rational Khovanov homology of the $(-3,5,7)$
    pretzel knot.}
\end{figure}

Figure~\ref{Kh357} depicts the Khovanov homology of $P(-3,5,7)$. It
shows an important qualitative feature of the general formulas we will
give: as discussed in the introduction, $Kh(P(-l,m,n))$ is made up of
a ``lower'' summand and an ``upper'' summand with respect to the $q$-
or $t$-grading. Out of the three allowable $\delta$-gradings, the
lower summand is contained in the highest two $\delta$-gradings, and
the upper summand is contained in the lowest two
$\delta$-gradings. For most pretzel links, the upper and lower
summands overlap in at most two $t$-gradings (although the $(-l,l,n)$
pretzel links have a larger overlap when $l$ is even).

We will now introduce a few definitions allowing us to efficiently
package our formulas, following the discussion in the
introduction. Each upper and lower summand is contained in two
adjacent $\delta$-gradings. Now, when a link has thin Khovanov
homology (i.e. homology contained in two $\delta$-gradings), the Lee
spectral sequence from \cite{Lee} tells us that the homology breaks into
``knight's moves'' and ``exceptional pairs'' (see
Figure~\ref{knightsep}). In our case, each of the summands $L$ and $U$
will individually break into knight's moves and exceptional
pairs. Thus we can specify the entire summand (up to overall grading)
by, first, specifiying what it is in the bottom $\delta$-grading, and
second, determining which generators in the bottom $\delta$-grading
are parts of exceptional pairs rather than knight's moves.

\begin{figure}
  \begin{center}
    \input{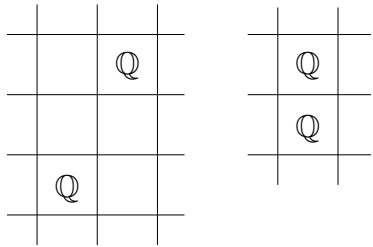}
  \end{center}
  \caption{\label{knightsep} A knight's move and an exceptional pair.}
\end{figure}

Specifying the upper or lower summand in its bottom $\delta$-grading
amounts (up to shifts) to specifying a list of dimensions along the
diagonal, or equivalently a sequence of integers. Since we will be
shifting the grading later, we may as well work at first with spaces
where the bottom $\delta$-grading is $\delta = 0$. These observations
motivate the following two definitions.

\begin{definition}\label{SequenceDef}
  If $a = (\ldots, a_{k-1}, a_k,a_{k+1},\ldots): \Z \to \N$ is a
  sequence of positive integers with only finitely many nonzero
  entries, define the bigraded vector space $\tilde{V}[a]$ associated
  to $a$ by
  \[
  \tilde{V}[a] = \sum_{i=-\infty}^{\infty} a_{i+1} q^{4i} t^{2i}.
  \]

  In other words, $\tilde{V}[a]$ is contained in $\delta$-grading
  zero, with ranks $(\ldots, a_{k-1}, a_k,a_{k+1},\ldots)$ along the
  diagonal. The term with rank $a_1$ has been placed at the
  origin. (We chose $a_1$ rather than $a_0$ to be the origin since the
  sequences used in most of our definitions will be supported on
  $\{1,2,\ldots,k\}$ for some $k$, and we would like the lowest
  nontrivial generator to have bigrading $(0,0)$.)
\end{definition}

\begin{figure}
  \begin{center}
    \input{sequences.pstex_t}
  \end{center}
  \caption{\label{sequences} Let $a$ be supported on $[1,4]$, with
    values $(1,2,3,4)$. On the left is $\tilde{V}[a]$, and on the
    right is $V[a,E]$, where there is one exceptional pair on the
    third index.}
\end{figure}

The summands $L$ and $U$ will be grading-shifts of spaces $L_{l,m,n}$
and $U_{l,m,n}$; the spaces $L_{l,m,n}$ and $U_{l,m,n}$ will be
contained in $\delta \in \{0,2\}$ (recall that these $\delta$-gradings
are adjacent, since differences in $q$-gradings or $\delta$-gradings
are always even for a given link). To assemble these spaces, we need
to know where their exceptional pairs are. Suppose we would like to
specify $V$, a bigraded vector space made up of knight's moves and
exceptional pairs and contained in $\delta \in \{0,2\}$. Besides
needing a sequence as in Definition~\ref{SequenceDef}, we also need
some exceptional pair data. We will package the exceptional pair data
of $V$ in the form of a function $E: \Z \rightarrow \N$, where $E(i)$
is the number of exceptional pairs $V$ has in $t$-grading $i$.

\begin{definition}
  Let $a$ be a finite sequence as in Definition~\ref{SequenceDef} and
  let $E: \Z \rightarrow \N$ be any function such that $E(i) \leq
  a_i$. Define the sequence $a'$ by $a'_i = a_i-E(i).$ Then
  \[
  V[a,E] := (1+q^4 t) \tilde{V}[a'] \oplus_{i=1}^{k} E(i) (1+q^2)
  q^{4(i-1)} t^{2(i-1)}
  \]

  In other words, we have singled out $E(i)$ generators in each index
  $i$ and turned them into the bottom halves of exceptional
  pairs. Each other generator has been turned into the bottom half of
  a knight's move pair.

  The definitions of the spaces $L_{l,m,n}$ and $U_{l,m,n}$ will
  involve sequences supported on $[1,k] := \{1,2,\ldots,k\}$ for some
  $k$. In these cases, we will specify $E: [1,k] \to \N$ by saying
  (for instance) ``the exceptional pair is on the first index'' or
  ``there are two exceptional pairs, one on the first index and one on
  the second-to-last index.'' If we wish to indicate that there are no
  exceptional pairs, we will simply omit $E$ from the notation and
  just write
  \[
  V[a] := (1 + q^4 t) \tilde{V}[a].
  \]
\end{definition}

\begin{definition}
  The following basic sequences will arise in our description of the
  upper and lower summands. All will be supported on $[1,k]$ for some
  $k$, and we will specify their values using $k$-tuples of positive
  integers.

  \begin{itemize}
  \item The sequence $a_k = (a_{1k}, \ldots, a_{kk})$ of length $k$
    has the following pattern:

    $(1,0,2,1,3,2,\ldots)$.
  \item The sequence $b_k = (b_{1k}, \ldots, b_{kk})$ of length $k$
    has the following pattern:

    $(1,0,2,0,3,1,4,2,5,3,\ldots)$.
  \item The sequence $c_k = (c_{1k}, \ldots, c_{kk})$ of length $k$
    has the following pattern:

    $(1,1,2,2,3,3,\ldots)$.
  \end{itemize}

  The sequences above are assumed to be truncated after the $k^{th}$
  entry, even if $k$ is odd. If $k \leq 0$, the sequences are empty.
\end{definition}

\begin{definition}
  We will also use some operations on sequences supported on $[1,k]$
  for various $k$. As before, such sequences will be specified with
  tuples of positive integers.
 
  \begin{itemize}
  \item If $a = (a_1,\ldots, a_k)$ and $b = (b_1,\ldots, b_l)$, then
    $a \cdot b$ denotes the concatenation $(a_1, \ldots, a_k, b_1,
    \ldots, b_l)$. Note that $a \cdot b$ is supported on $[1,k+l]$.

  \item If $a$ is as above, $a^m$ denotes $a \cdot a \cdot \ldots
    \cdot a$, $m$ times. For example, $(1)^m$ denotes the sequence
    $(1,\ldots,1)$ of length $m$.

  \item If $a$ is as above, $\overline{a}$ denotes $a$ in reverse,
    i.e. $(a_k, \ldots, a_1)$.

  \item If $a$ and $b$ are as above and $k=l$, then we can add $a$ and
    $b$ componentwise to get $a+b$. This makes sense even if $b$ has
    negative coefficients, as long as the result has positive
    coefficients.

  \end{itemize}
\end{definition}

\subsection{Formula for the lower summand.} After these preliminaries,
we can now define the space $L_{l,m,n}$ which will be the ``lower''
summand of $Kh(P(-l,m,n))$ after a grading shift. For $m \neq l$,
there are two cases of the definition, depending only on the parity of
$l$. There are four additional cases for $m = l$. The formulas are
given below in the notation of the above subsection.

\begin{figure}
  \begin{center}
    \input{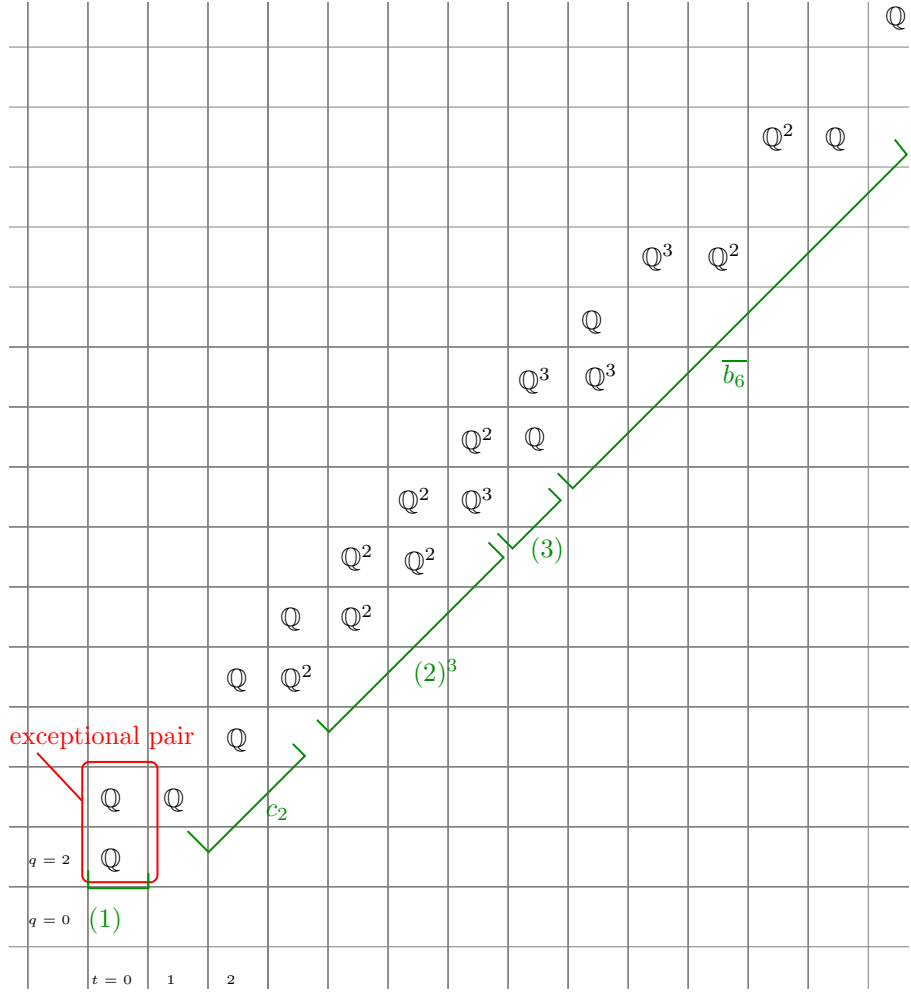}
  \end{center}
  \caption{\label{L6example} The space $L_{6,m,n}$ for any $m > 6$,
    based on the formula in case \ref{evenmneql} of
    Definition~\ref{lower}.}
\end{figure}

\begin{definition}\label{lower} The lower summand $L_{l,m,n}$ is
  defined below in various cases. See Figure~\ref{L6example} for an
  example of case \ref{evenmneql} below.
  \begin{enumerate}
  \item \label{oddmneql} If $m \neq l$ and $l$ is odd, then
    \[
    L_{l,m,n} := V \bigg[ a_{l-1} \cdot \bigg( \frac{l-1}{2} \bigg)^2
    \cdot \overline{a_{l-1}} \bigg].
    \]

  \item \label{evenmneql} If $m \neq l$ and $l$ is even, then
    \[
    L_{l,m,n} := V \bigg[ (1) \cdot c_{l-4} \cdot \bigg( \frac{l-2}{2}
    \bigg)^3 \cdot \bigg( \frac{l}{2} \bigg) \cdot \overline{b_l}, E
    \bigg],
    \]
    where the exceptional pair is on the first index.

  \item \label{oddmeql} If $m = l$ and $l$ is odd, then
    \[
    L_{l,l,n} := V \bigg[ a_{l-1} \cdot \bigg (\frac{l-1}{2} \bigg)^2
    \cdot \overline{a_{l-1}} \cdot (0)^{n-l} \cdot (1),E \bigg],
    \]
    where the exceptional pair is on the final index.

  \item \label{evenmeqlodd} If $m = l$, $l$ is even, and $n$ is odd,
    then
    \[
    L_{l,l,n} := V \bigg[(1) \cdot c_{l-4} \cdot \bigg( \frac{l-2}{2}
    \bigg)^3 \cdot \bigg( \frac{l}{2} \bigg) \cdot \overline{b_l}
    \cdot (0,1)^{(n-l-1)/2}, E \bigg],
    \]
    where the exceptional pair is on the first index.

  \item \label{evenmeqlevenneq} If $m = l$, $l$ is even, $n$ is even,
    and $n \neq l$, then
    \[
    L_{l,l,n} := V \bigg[(1) \cdot c_{l-4} \cdot \bigg (\frac{l-2}{2}
    \bigg)^3 \cdot \bigg(\frac{l}{2} \bigg) \cdot \overline{b_l} \cdot
    (0,1)^{(n-l)/2}, E \bigg],
    \]
    with exceptional pairs on the first and last indices.

  \item \label{llleven}If $m = l$, $l$ is even, and $n = l$, then
    \[
    L_{l,l,l} := V \bigg[(1) \cdot c_{l-4} \cdot \bigg (\frac{l-2}{2}
    \bigg)^3 \cdot \bigg(\frac{l}{2} \bigg) \cdot (\overline{b_l} +
    (\ldots,0,0,1)), E \bigg],
    \]
    with one exceptional pair on the first index and two on the last
    index. The addition is done such that the last index of
    $\overline{b_l}$ gets the extra $1$.

  \end{enumerate}
\end{definition}

Note that to obtain the formulas when $m = l$, you just add some extra
generators, in higher $q$- and $t$-gradings, to the formulas for $m
\neq l$.

\subsection{Formula for the upper summand.} The rest of the Khovanov
homology of $P(-l,m,n)$ comes from the ``upper summand''
$U_{l,m,n}$. It depends only on $g := m-l$ and $h := n-m$ as well as
the parities of $l$, $m$, and $n$. Each of the eight
choices for parities gives rise to a different formula, so the below
definition has eight different cases.

\begin{definition}\label{upper} Suppose $2 \leq l \leq m \leq n$ are
  integers; let $g = m-l$ and $h = n-m$.

  \begin{enumerate}
  \item\label{upperooo} If $l$, $m$, and $n$ are odd,
    \[
    U_{l,m,n} := V \bigg[ a_{g} \cdot \bigg (\frac{g}{2} \bigg)^{h}
    \cdot \overline{a_{g}}, E \bigg],
    \]
    where the one exceptional pair is on the final index.

  \item\label{upperooe} If $l$ and $m$ are odd but $n$ is even,
    \[
    U_{l,m,n} := V \bigg[ a_{g} \cdot \bigg(\frac{g}{2} \bigg)^{h-1}
    \cdot \overline{c_{g}},E \bigg],
    \]
    where the exceptional pair is on the first index of
    $\overline{c_{g}}$.

  \item\label{upperoeo} If $l$ is odd, $m$ is even, and $n$ is odd,
    \[
    U_{l,m,n} := V \bigg[ a_{g-1} \cdot \bigg(\frac{g+1}{2},
    \frac{g-1}{2} \bigg)^{(h+1)/2} \cdot \overline{c_{g-1}},E \bigg],
    \]
    where the exceptional pair is on the first instance of
    $\frac{g+1}{2}$.

  \item\label{upperoee} If $l$ is odd but $m$ and $n$ are both even,
    \[
    U_{l,m,n} := V \bigg[ a_{g-1} \cdot \bigg (\frac{g+1}{2},
    \frac{g-1}{2} \bigg)^{h/2} \cdot \bigg( \frac{g+1}{2} \bigg) \cdot
    \overline{c_{g-1}},E \bigg],
    \]
    where the exceptional pairs are on the first and last instances of
    $\frac{g+1}{2}$. (If $h = 0$, this means there are two exceptional
    pairs in the same $t$-grading.)

  \item\label{uppereoo} If $l$ is even but $m$ and $n$ are both odd,
    \[
    U_{l,m,n} := V \bigg[ b_{g+1} \cdot \bigg( \frac{g+1}{2},
    \frac{g-1}{2} \bigg)^{h/2} \cdot \overline{c_{g-1}} \bigg].
    \]

  \item\label{uppereoe} If $l$ is even, $m$ is odd, and $n$ is even,
    \[
    U_{l,m,n} := V \bigg[ b_{g+1} \cdot \bigg(\frac{g+1}{2},
    \frac{g-1}{2} \bigg)^{(h-1)/2} \cdot \overline{c_g}, E \bigg],
    \]
    where the exceptional pair is on the first index of
    $\overline{c_g}$.

  \item\label{uppereeo} If $l$ and $m$ are even but $n$ is odd,
    \[
    U_{l,m,n} := V \bigg[ b_{g} \cdot \bigg(\frac{g+2}{2},
    \frac{g-2}{2} \bigg)^{(h+1)/2} \cdot \overline{c_{g-1}}, E \bigg],
    \]
    where the exceptional pair is on the first instance of
    $\frac{g+2}{2}$. If $l = m$, so $g = 0$, the $-1$ that arises here
    should be interpreted as zero.

  \item\label{uppereee} If $l$, $m$ and $n$ are all even,
    \[
    U_{l,m,n} := V \bigg[ b_g \cdot \bigg(\frac{g+2}{2}, \frac{g-2}{2}
    \bigg)^{h/2} \cdot \bigg( \frac{g+2}{2} \bigg) \cdot
    \overline{a_g}, E \bigg],
    \]
    where the exceptional pairs are on the first and last instances of
    $\frac{g+2}{2}$ and on the very last index. Again, if $g = 0$, the
    $-1$ in this formula should be interpreted as zero.
  \end{enumerate}
\end{definition}

\subsection{Orientations, grading shifts, and the general formula.}
Finally, we will put everything together with the appropriate grading
shifts to produce the general formula for the Khovanov homology of
pretzel links.

\begin{figure}
  \begin{center}
    \epsfbox{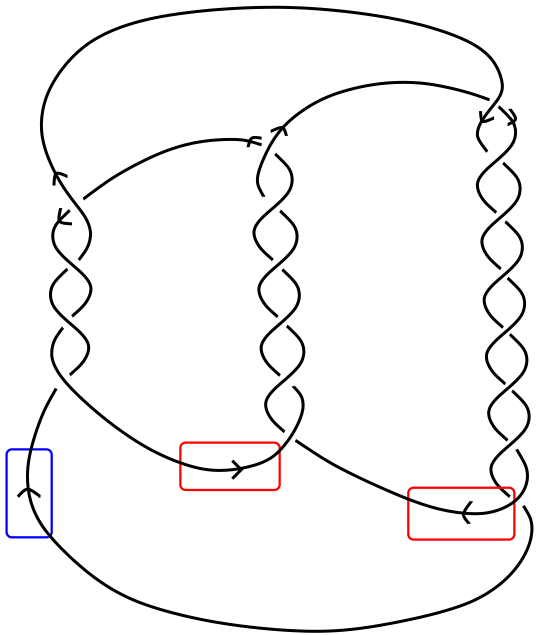}
  \end{center}
  \caption{\label{Pn468} The link $P(-4,6,8)$, or $P(-4,6,8)_{RL}$,
    with the orientations we will use for even $l$. Note that all
    crossings are positive; this is a peculiarity of the $RL$
    orientation pattern.}
\end{figure}

Before discussing grading shifts, though, we must decide on
orientations, since changing the orientation of one component of a
multi-component link changes the Khovanov homology by an overall
grading shift.

The colored boxes in Figure~\ref{Pn468} indicate the data needed to
specify the orientation of $P(-l,m,n)$ in each case; for $3$-component
links like $P(-4,6,8)$, we can pick a direction in each box. For the
pretzel links which are knots or $2$-component links, not all choices
of this data are allowable.

We will always orient the blue box upwards as shown in
Figure~\ref{Pn468}. For knots, this choice fixes the entire
orientation, according to Proposition~\ref{knotor} below. For links,
we need to pin down the red boxes too. Each can point either right or
left. We will indicate the way they point with subscripts. For
example, when $P(-4,6,8)$ is oriented as in Figure~\ref{Pn468}, we
will write it as $P(-4,6,8)_{RL}$.

When $P(-l,m,n)$ is a knot, the directions of the red boxes are fixed
by our choice for the blue box:

\begin{proposition}\label{knotor} Using the above notation:
  \begin{itemize}
  \item If $l$, $m$, and $n$ are odd, then $P(-l,m,n)$ is oriented
    $RR$.
  \item If $l$ and $m$ are odd but $n$ is even, then $P(-l,m,n)$ is
    oriented $LR$.
  \item If $l$ is odd, $m$ is even, and $n$ is odd, then $P(-l,m,n)$
    is oriented $LL$.
  \item If $l$ is even but $m$ and $n$ are odd, then $P(-l,m,n)$ is
    oriented $RL$.
  \end{itemize}
\end{proposition}

When dealing with $P(-l,m,n)$ for even $l$, we will always orient the
link in the $RL$ manner as in Figure~\ref{Pn468} (and omit the
subscripts $RL$), in agreement with the orientation given by
Proposition~\ref{knotor} when $m$ and $n$ are odd. This choice fixes
orientations on all the links under consideration, except for
$P(-l,m,n)$ when $l$ is odd and $m$ and $n$ are both even. In this
case, our inductive proofs will force us to consider both
$P(-l,m,n)_{LL}$ and $P(-l,m,n)_{LR}$, and here we will be careful to
indicate which one we mean.

\begin{proposition}\label{GradingShiftProp}
  The values of the grading shift variables $\sigma_L$, $\tau_L$,
  $\sigma_U$, and $\tau_U$ in Theorem~\ref{intromaintheorem} depend
  only on the orientation pattern ($RR$, $LL$, $RL$, or $LR$) and are
  listed in Table~\ref{GradingShift}.
\end{proposition}

\begin{table}
  \begin{center}
    \begin{tabular}{|c||c|c|c|c|c|c|}
      \hline
      & $\sigma_L$ & $\tau_L$ & $\sigma_U$ & $\tau_U$ & $\delta_{max}$ & $n_-$ \\
      \hline\hline
      $RR$ & $-2m-2n-1$ & $-m-n$ & $4l-2m-2n-1$ & $2l-m-n+1$ & $1$ & $m+n$ \\
      \hline
      $LL$ & $-3l-2m+n-1$ & $-l-m$ & $l-2m+n-1$ & $l-m+1$ & $n-l+1$ & $l+m$ \\
      \hline
      $LR$ & $-3l+m-2n-1$ & $-l-n$ & $l+m-2n-1$ & $l-n+1$ & $m-l+1$ & $l+n$ \\
      \hline
      $RL$ & $n+m-1$ & $0$ & $4l+m+n-3$ & $2l$ & $n+m+1$ & $0$ \\
      \hline
    \end{tabular}
  \end{center}
  \caption{Values of the grading shift variables for the four different
    orientation possibilities. For convenience, the rightmost two columns
    also list the highest of the three $\delta$-gradings with nonzero homology
    and the number $n_-$ of negative crossings.}
  \label{GradingShift}
\end{table}

At this point we may restate our main theorem, having defined all of
its components:

\begin{theorem}\label{maintheorem}
  Suppose $2 \leq l \leq m \leq n$. Then
  \[
  Kh(P(-l,m,n)) = q^{\sigma_L} t^{\tau_L} L_{l,m,n} \oplus
  q^{\sigma_U} t^{\tau_U} U_{l,m,n},
  \]
  where $L_{l,m,n}$ comes from Definition~\ref{lower}, $U_{l,m,n}$
  comes from Definition~\ref{upper}, and the values of $\sigma_L$,
  $\tau_L$, $\sigma_U$, and $\tau_U$ come from
  Proposition~\ref{GradingShiftProp}.
\end{theorem}

\begin{definition}\label{LUDef}
  As above, we will sometimes write $Kh(P(-l,m,n)) = L \oplus U$. This
  means $L := q^{\sigma_L} t^{\tau_L} L_{l,m,n}$ and $U :=
  q^{\sigma_U} t^{\tau_U} U_{l,m,n}$ as in
  Theorem~\ref{maintheorem}. We will also say that $L$ is ``based at''
  $(\tau_L, \sigma_L)$ and $U$ is ``based at'' $(\tau_U, \sigma_U)$.
\end{definition}

\begin{corollary}\label{DeltaCorollary}
  $Kh(P(-l,m,n))$ is contained in three $\delta$-gradings which depend
  only on the orientation pattern ($RR$, $LL$, $LR$, or $RL$). These
  gradings are $\delta_{max}$, $\delta_{max}-2$, and $\delta_{max}-4$,
  where $\delta_{max}$ is given in Table~\ref{GradingShift}.
\end{corollary}

\subsection{An alternative approach to the grading shift data.}
\label{sec:an-altern-appr}

The values in Table~\ref{GradingShift} may seem a bit mysterious. The
remainder of this section will discuss an alternative way of
specifying the values of $\sigma_L$, $\tau_L$, $\sigma_U$, and
$\tau_U$ in this table.

Write $Kh(P(-l,m,n)) = L \oplus U$ as in Definition~\ref{LUDef}. We
could ask how much the $t$-gradings of $L$ and $U$ overlap. As it
turns out, this difference follows a simple pattern. For the purpose
of this section, we only care about the generic case $m \neq l$:

\begin{proposition}\label{NewOverlapProp}
  With the above notation, let $\Delta$ be the smallest $t$-grading in
  which $U$ has a nonzero generator, minus the highest $t$-grading in
  which $L$ has a nonzero generator. Suppose $m \neq l$. If $l$ is
  odd, then $\Delta = 1$. If $l$ is even, then $\Delta = -1$.
\end{proposition}

We will take the values of $\delta_{max}$ in Table~\ref{GradingShift}
as given. Since $L$ is contained in the top two $\delta$-gradings and
$U$ is contained in the bottom two $\delta$-gradings,
Proposition~\ref{NewOverlapProp} pins down how the $L$ summand relates
to the $U$ summand. All that is left is to fix one overall reference
point.

In other words, knowing the $q$- and $t$-gradings of any generator of
the upper or lower summand will suffice to pin down the overall
gradings of the Khovanov homology. Luckily, though, the $t$-gradings
of exceptional pairs are easily computable through linking-number data:

\begin{proposition}[Lee \cite{LeeEndo}]\label{linking}
  If $L$ is a knot, then $Kh(L)$ has an exceptional pair in $t =
  0$. If $L$ is a $2$-component link with components $L_1$ and $L_2$,
  then $Kh(L)$ has exceptional pairs in $t = 0$ and $t =
  2lk(L_1,L_2)$. If $L$ is a $3$-component link with components $L_1$,
  $L_2$, and $L_3$, then $Kh(L)$ has exceptional pairs in $t = 0$, $t
  = 2(lk(L_1,L_2) + lk(L_1,L_3))$, $t = 2(lk(L_1,L_2) + lk(L_2,L_3))$,
  and $t = 2(lk(L_1,L_3) + lk(L_2,L_3))$. These are all the
  exceptional pairs in $Kh(L)$.
\end{proposition}

Once we know the $t$-grading of a generator of an exceptional pair, we
can pin down its $q$-grading by knowing its $\delta$-grading. But we
know whether each exceptional pair is in $L$ or in $U$, since this
data was included in the definitions of $L$ and $U$, and we know the
$\delta$-gradings of generators of $L$ and $U$. Hence, given the
values of $\delta_{max}$ in Table~\ref{GradingShift} and the overlap
data in Proposition~\ref{NewOverlapProp}, we can populate the rest of
Table~\ref{GradingShift} simply from the definitions of $L_{l,m,n}$
and $U_{l,m,n}$.

For example, consider the top row of Table~\ref{GradingShift}. Looking
at $Kh(P(-l,m,n))$ for odd values of $\{l,m,n\}$ with $m \neq l$ will
be enough to fill in this row, since for these values $P(-l,m,n)$ is
oriented $RR$. In this case, the summand $L$ fills $2l+1$ columns of
the grid, because the sequence defining $L_{l,m,n}$ has length $2l$
and the knight's move pair on the far right end spills over into the
next column. The summand $U$ fills $-2l+m+n$ columns (the sequence
defining $U_{l,m,n}$ in this case has length $2g+h = -2l+m+n$, and
this time there is an exceptional pair on the far right).

Proposition~\ref{NewOverlapProp} tells us that we should put the
columns of $U$ immediately to the right of the columns of $L$. The one
exceptional pair is on the last column of $U$, so this column must be
$t = 0$. Hence the first column of $U$ is $t = 2l-m-n+1$, and so
$\tau_U = 2l-m-n+1$. The one generator of $U$ in this leftmost column
is in $\delta = -3$, so its $q$-grading must be $4l-2m-2n-1$. Hence
$\sigma_U = 4l-2m-2n-1$.

The summand $L$ occupies $2l+1$ columns to the left of $t = 2l-m-n+1$,
so the leftmost column of $L$ is $t = -m-n$. The generator of $L$ in
this column lies in $\delta = -1$, so its $q$-grading is
$-2m-2n-1$. Hence $\tau_L = -m-n$ and $\sigma_L = -2m-2n-1$. The rest
of the table can be completed similarly.

\section{Preliminaries for the proof}

\subsection{Skein sequences and cancellations.}

Our main computational tool will be the unoriented skein exact
sequence in Khovanov homology, stated below.
\begin{theorem}\label{skein}{\em (See \cite{KPKH}.)}
  Let $D$ be a diagram for an oriented link, and consider a crossing
  $c$. One of the two resolutions of $c$, say $D_o$, is consistent
  with the orientations, and will be called the ``oriented
  resolution.'' One, say $D_u$, is not (the ``unoriented''
  resolution). Let $\epsilon = n_-(D_u) - n_-(D)$, where $n_-$ denotes
  the number of negative crossings in a diagram. Then, if $c$ is a
  positive crossing, we have the sequence
  \[
  \xymatrix{\cdots \ar[r]^-{f} & q^{3\epsilon + 2} t^{\epsilon + 1}
    Kh(D_u) \ar[r] & Kh(D) \ar[r] & q Kh(D_o) \ar[r]^-{f} & \cdots}
  \]

  If $c$ is a negative crossing, we have the sequence
  \[
  \xymatrix{\cdots \ar[r]^-{f} & q^{-1} Kh(D_o) \ar[r] & Kh(D) \ar[r]
    & q^{3\epsilon + 1} t^{\epsilon} Kh(D_u) \ar[r]^-{f} & \cdots.}
  \]
\end{theorem}

Schematically, the skein exact sequence will put us in the following
situation: we have two known bigraded vector spaces $V$ and $W$ and a
map $f: V \to W$ fitting in an exact sequence:
\[
\xymatrix{\cdots \ar[r]^f & W \ar[r] & X \ar[r] & V \ar[r]^f & W
  \ar[r] & \cdots,}
\]
where $X$ is unknown. Our goal will be to determine $X$. We
know that $f$ preserves $q$-grading and increases $t$-grading by
one. Thus, $X$ arises from ``cancelling'' pairs of generators from $V
\oplus W$ as in the following definition:

\begin{definition}
  Let $V$ and $W$ be bigraded vector spaces. A cancellation of $V
  \oplus W$ is a subspace $X$ of $V \oplus W$ obtained by eliminating
  ``horizontal pairs,'' i.e. two-dimensional subspaces $\Q v \oplus \Q
  w$ where $v \in V$, $w \in W$, $v$ and $w$ have the same
  $q$-grading, and the $t$-grading of $w$ is one greater than the
  $t$-grading of $v$. For an example, see Figure~\ref{stdcancel}.
\end{definition}

Hence we can determine $X$ by looking at all possible cancellations of
$V \oplus W$ and rejecting all but one of them. We will accomplish
this task by using the structure of Khovanov homology coming from the
Lee spectral sequence (see \cite{Lee}). For links with Khovanov homology
contained entirely in three adjacent $\delta$-gradings, this spectral
sequence implies that the Khovanov homology breaks up into the
knight's moves and exceptional pairs discussed earlier (see
Figure~\ref{knightsep}). Motivated by this fact, we make the following
definition.

\begin{definition}\label{wellstructured}
  A bigraded vector space $V$ is well-structured if it is a sum of
  knight's moves and exceptional pairs.
\end{definition}

For our sequences, all three spaces $V$, $W$, and $X$ will be
well-structured. Thus, in trying to determine $X$, we can first
disregard all cancellations of $V \oplus W$ which are not
well-structured. Out of the remaining options, it will turn out that
we can uniquely determine the correct choice of $X$ by looking at the
number and placement of exceptional pairs. This general strategy will
be implemented in the proofs below.

We can save ourselves some work with a general lemma, for which we
need notation related to that of Section~\ref{genform}:

\begin{definition}\label{SequenceAndE}
  Let $V$ be a well-structured vector space contained in two adjacent
  $\delta$-gradings. Let $E_V: \Z \rightarrow \N$ be the function such
  that $E_V(i)$ is the number of exceptional pairs of $V$ in
  $t$-grading $i$, and let $a_V$ be the sequence such that $V =
  V[a_V,E_V]$.  Note that $a_V$ is supported on the ``actual''
  $t$-gradings of $V$, not necessarily on $[1,k]$ for any $k$.
\end{definition}

\begin{remark}\label{nodepend}
  Once $E_V$ is well-defined, it is clear that $a_V$ is well-defined,
  by subtracting values of $E_V$ from ranks of $V$. However, it is a
  bit tricky to see why $E_V$ is well-defined. One way to do this is
  as follows: look at the highest $t$-grading $i$ of $V$. Note that
  $V$ is contained in two $\delta$ gradings, so there are two possible
  $q$-gradings, say $j$ and $j+2$, corresponding to $t = i$. Then
  $E_V(k) = 0$ for $k > i$ and $E_V(i) = \dim V_{i,j}$.  Now that we
  know $E_V(i)$, we can look at $t$-grading $i-1$: $\dim V_{i-1,j-2} -
  E_V(i-1) = \dim V_{i,j+2} - E_V(i)$, since both count the number of
  knight's move pairs whose lower generator lies in bigrading
  $(i-1,j-2)$. Hence we can deduce the value of $E_V(i-1)$. Continuing
  this process, $E_V$ is well-defined. Furthermore, it is clear that
  $E_V(i)$ depends only on $V_{k,*}$ for $k \geq i$.

  One could equivalently begin with the lowest $t$-grading rather than
  the highest one, and see that $E_V(i)$ alternatively depends only on
  $V_{k,*}$ for $k \leq i$.
\end{remark}

\begin{lemma}\label{noeps}
  In the above situation, suppose $V$ and $W$ are contained in the
  same two adjacent $\delta$-gradings, and suppose also that $E_X \geq
  E_V + E_W$. Then in fact $X = V \oplus W$, i.e. no cancellations are
  possible.
\end{lemma}

\begin{proof}
  Suppose $X$ is some nontrivial well-structured cancellation. Choose
  a cancelling pair of generators $(e \in V, f \in W)$ in the highest
  possible $q$-grading. Say the bigrading of $f$ is $(i,j)$.

  By well-structuredness of $V$ and $W$, we have the following:
  \begin{align*}
    &\dim V_{i,j} - E_V(i) = \dim V_{i+1,j+4} - E_V(i+1); \\
    &\dim W_{i,j} - E_W(i) = \dim W_{i+1,j+4} - E_W(i+1); \\
    &\dim X_{i,j} - E_X(i) = \dim X_{i+1,j+4} - E_X(i+1).
  \end{align*}
  But since $(e,f)$ cancels, we have
  \begin{equation}\label{lessthan}
    \dim X_{i,j} < \dim V_{i,j} + \dim W_{i,j}.
  \end{equation}
  Since $(e,f)$ is the highest pair to cancel, $\dim X_{i+1,j+4} =
  \dim V_{i+1,j+4} + \dim W_{i+1,j+4}$. By Remark~\ref{nodepend},
  $E_X(i+1) = E_V(i+1) + E_W(i+1)$ since $X = V \oplus W$ in
  $t$-gradings $\geq i+1$. Hence, making the appropriate substitutions
  and cancellations in \eqref{lessthan}, we obtain
  \[
  E_X(i) < E_V(i) + E_W(i),
  \]
  contradicting our assumption.
\end{proof}

In particular, if $V$ and $W$ have no exceptional pairs (i.e. $E_V$
and $E_W$ are identically zero), then there can be no nontrivial
well-structured cancellation of $V \oplus W$.

One type of cancellation which occurs frequently is the following
``standard cancellation:'' suppose $V$ and $W$ are contained in the
same two adjacent $\delta$-gradings, with $E_V(i) > 0$ and $E_W(i+1) >
0$. Then one can single out two generators each from $V$ and $W$ in
the configuration of Figure~\ref{stdcancel}. Cancelling the generators
in the middle $q$-grading produces a knight's move, so the resulting
space is still well-structured.

\begin{figure}
  \begin{center}
    \input{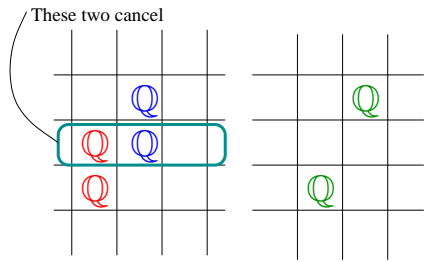}
  \end{center}
  \caption{\label{stdcancel} A standard cancellation. Red generators
    are from $V$, blue generators are from $W$, and green generators
    belong to $X$ after the cancellation.}
\end{figure}

\begin{remark}\label{nomore}
  Below, we will often see a set of cancellations containing several
  standard cancellations, and we will want to show no more
  cancellations can occur. We can argue as follows: suppose we did
  some standard cancellations on $V \oplus W$ to get $X$, and then did
  another cancellation. The additional cancellation would, in fact, be
  a cancellation of $V' \oplus W'$, a proper subspace of $X$ obtained
  from $V \oplus W$ by removing all four generators at each standard
  cancellation (rather than only those in the middle $q$-grading). The
  reason is that, after a standard cancellation, the two remaining
  generators under consideration are in the wrong $\delta$-gradings to
  cancel (the generator of $V$ is in the lower $\delta$-grading, and
  the generator of $W$ is in the higher one). Now, the spaces $V'$ and
  $W'$ are well-structured and have fewer exceptional pairs than $V$
  and $W$, so in many cases we will be able to apply Lemma~\ref{noeps}
  and derive a contradiction.
\end{remark}

We can also identify a situation in which we can conclude a standard
cancellation occurred, given some information about $X$:

\begin{lemma}\label{lstandard}
  Assume $V$ and $W$ are contained in the same two adjacent
  $\delta$-gradings. Suppose $i$ is the lowest $t$-grading of $W$ and
  $E_W(i) > 0$. Assume also that $E_V(i-1) > E_X(i-1)$. Then a
  standard cancellation must have occurred between $t$-gradings $i-1$
  and $i$.
\end{lemma}

\begin{proof}
  Let $j$ be the lowest $q$-grading of $W$. As in the previous lemma,
  we have $\dim V_{i-2,j-4} - E_V(i-2) = \dim V_{i-1,j} - E_V(i-1)$
  and $\dim X_{i-2,j-4} - E_X(i-2) = \dim X_{i-1,j} - E_X(i-1)$. We
  also know $V_{i-2,j-4} = X_{i-2,j-4}$ since $i$ is the lowest
  $t$-grading of $W$.

  Suppose no standard cancellation occurred; then $\dim X_{i-1,j} =
  \dim V_{i-1,j}$. After some substitutions, this equation becomes
  $E_X(i-1) - E_X(i-2) = E_V(i-1) - E_V(i-2)$. But since $X = V$ in
  $t$-gradings $\leq i-2$, we have $E_V(i-2) = E_X(i-2)$ by
  Remark~\ref{nodepend}. Thus $E_X(i-1) = E_V(i-1)$, a contradiction.
\end{proof}

By interchanging $V$ with $W$ and replacing $i-1$ with $i+1$, we get
the following:

\begin{lemma}\label{hstandard}
  Suppose $i$ is the highest $t$-grading of $V$ and $E_V(i) >
  0$. Assume also that $E_W(i+1) > E_X(i+1)$.  Then a standard
  cancellation must have occurred between $t$-gradings $i$ and $i+1$.
\end{lemma}

\subsection{Jones polynomial calculation.}
We need a Jones polynomial calculation to begin our inductive
proofs. To fix notation, if $L$ is a link, then $V_L(q^2)$ will denote
the (normalized) Jones polynomial of $L$. (As usual, we will write
everything in terms of the variable $q$, which squares to the standard
argument of the Jones polynomial.) The unnormalized Jones polynomial
$\overline{V}_L(q^2)$ is defined to be $(q + q^{-1})V_L(q^2)$.

\begin{lemma}\label{jones}
  Let $l \geq 2$. The unnormalized Jones polynomial of the link
  $P(-l,l,0)_{LR}$ is
  \begin{align*}
    \overline{V}_{P(-l,l,0)_{LR}}(q^2) &= (-1)^l q^{2l+1} +
    \sum_{j=1}^{l-3} (-1)^{j+l+1} q^{2l-2j-1} + (2q + 2q^{-1}) \\
    &+ \sum_{i=1}^{l-3} (-1)^{i+1} q^{-2i-3} +(-1)^l q^{-2l-1}.
  \end{align*}
\end{lemma}

\begin{remark}
  When $l$ is odd, the orientation $LR$ is forced by the choices we
  made earlier. When $l$ is even, we will need to consider
  $P(-l,l,0)_{RL}$, whose Jones polynomial differs from that of
  $P(-l,l,0)_{LR}$ by an overall factor of $q$. To pin down this
  factor, note that $P(-l,l,0)_{RL}$ is (after flipping the diagram
  over) just $P(-l,l,0)_{LR}$ with the orientation of the ``outer''
  component reversed. The linking number of this component with the
  rest of the link in $P(-l,l,0)_{LR}$ is $-l/2$. Hence
  $P(-l,l,0)_{RL}$ picks up a factor of $q^{-6(l/2)} = q^{3l}$. The
  resulting formula for the unnormalized Jones polynomial of the link
  $P(-l,l,0)_{RL}$ is
  \begin{align}\label{RLjones}
    \overline{V}_{P(-l,l,0)_{RL}}(q^2) &= (-1)^l q^{5l+1} + \sum_{j=1}^{l-3} (-1)^{j+l+1} q^{5l-2j-1} + (2q^{3l+1} + 2q^{3l-1}) \\
    &+ \sum_{i=1}^{l-3} (-1)^{i+1} q^{3l-2i-3} +(-1)^l
    q^{l-1}. \nonumber
  \end{align}
\end{remark}

\begin{proof}[Proof of Lemma~\ref{jones}]
  The link $P(-l,l,0)_{LR}$ is a connected sum of the positive
  right-handed $(2,l)$ torus link $T_{2,l}$ and its mirror
  $\overline{T_{2,l}}$, the left-handed negatively oriented $(2,l)$
  torus link. Since the Jones polynomial is multiplicative under
  connected sum, we can obtain the Jones polynomial of
  $P(-l,l,0)_{LR}$ easily. We start with the well-known formula
  $V_{T(2,l)}(q^2) = q^{l-1} +
  \sum_{i=1}^{l-1}(-1)^{i+1}q^{l+2i+1}$. Therefore
  \begin{align*}
    V_{P(-l,l,0)_{LR}}(q^2)&= \bigg( q^{-l+1}+\sum_{i=1}^{l-1}(-1)^{i+1}q^{-l-2i-1} \bigg) \bigg( q^{l-1} + \sum_{j=1}^{l-1}(-1)^{j+1}q^{l+2j+1} \bigg) \\
    &= 1 + \sum_{i=1}^{l-1}(-1)^{i+1}q^{-2i-2} + \sum_{j=1}^{l-1}(-1)^{j+1}q^{2j+2} + \sum_{i=1}^{l-1} \sum_{j=1}^{l-1} (-1)^{i+j} q^{2j-2i} \\
    &= (-1)^l q^{2l} + \sum_{j=1}^{l-2} (-1)^{j+l}jq^{2l-2j} -(l-2) q^2 + l - (l-2)q^{-2} \\
    &+ \sum_{i=1}^{l-2}(-1)^{i+1}(l-i-1)q^{-2i-2} + (-1)^l q^{-2l}.
  \end{align*}

  Multiplying by $(q + q^{-1})$, we get the above formula for the
  unnormalized Jones polynomial.
\end{proof}

\subsection{Proof strategy.}\label{strategy}
We will now outline how the proof of the general formula will be
structured. Consider a crossing in the standard diagram for
$P(-l,m,n)$. One resolution of the crossing produces another pretzel
link, with either $l$, $m$, or $n$ reduced by one. The other
resolution produces a torus link whose Khovanov homology is
known. Hence, given an appropriate base case for induction, the skein
exact sequence of the crossing relates two known entities (the
Khovanov homology of a torus link and of a smaller pretzel link) with
the unknown entity we would like to compute.

Thus, the most naive idea for a proof might be to pick one strand of
$P(-l,m,n)$ and unravel it, one crossing at a time, until we reach a
quasi-alternating link whose Khovanov homology we know. One could hope
that at each step, the skein exact sequence provides enough data to
reduce the computation for the larger pretzel link to the computation
for the smaller one, inductively determining $Kh(P(-l,m,n))$.

Unfortunately, the sequence does not always contain enough data; there
are some ambiguities. For example, suppose we tried to unravel the
middle strand to reach the quasi-alternating link $P(-l,l-1,n)$. An
ambiguity would arise in trying to determine $P(-l,l,n)$ from
$P(-l,l-1,n)$. If $l$ or $n$ is even, a further ambiguity arises in
trying to determine $P(-l,l+1,n)$ from $P(-l,l,n)$.

One way around these ambiguities is to unravel the middle strand as
far as possible, and then unravel the rightmost strand until reaching
a quasi-alternating link. For odd $l$ and $n$, this amounts to a
series of reductions from $P(-l,m,n)$, to $P(-l,l,n)$, to
$P(-l,l,l-1)$. Everything in this procedure works, as we will see
below.

For even $l$ or even $n$, this strategy would mean going from
$P(-l,m,n)$, to $P(-l,l+1,n)$, to $P(-l,l+1,l-1)$. If $l$ is even,
each step works out. However, if $l$ is odd and $n$ is even, another
ambiguity arises: the skein sequence does not uniquely determine
$Kh(P(-l,l+1,l+2))$ from $Kh(P(-l,l+1,l+1))$. Luckily, though, by this
time we already know $P(-l,l+1,n)$ for all odd $n$. Hence we only need
to go from $P(-l,l+1,n)$ to $P(-l,l+1,n-1)$ when $n$ is even. In this
case the skein exact sequence does give us enough data.

We will organize the proof as follows: first, we will consider the
case when $l$ is odd. We will prove the formula for $P(-l,l,n)$ and
then complete the proof for the case of odd $l$ and $n$. We will next
deduce the formula for $P(-l,l+1,n)$ (even $n$) from the formula for
$P(-l,l+1,n-1)$. Then we will derive the general formula for
$P(-l,m,n)$ with odd $l$.

For even $l$, the roadmap is a bit simpler. We will prove the formulas
for $P(-l,l,n)$ and $P(-l,l+1,n)$ first, and then deduce the general
formula for $P(-l,m,n)$.

\section{Proof of the general formula for odd $l$}

\subsection{$P(-l,l,n)$ for odd $l$.}

We begin with the special case $m = l$. A glance at
Definition~\ref{upper} reveals that $U_{l,l,n} = 0$ for odd $l$.
Corollary~\ref{DeltaCorollary} tells us we should be proving that the
Khovanov homology lies in $\delta = 1$ and $\delta = -1$, with the
form specified in Definition~\ref{lower}. More precisely, we have the
following lemma, which holds for all $n$ (not just $n \geq l$):

\begin{lemma}\label{oddhalfway} 
  Let $l \geq 3$ be odd and let $n \geq 0$. Let $a^{(n)}$ denote the
  sequence $a_{l-1} \cdot (\frac{l-1}{2})^2 \cdot \overline{a_{l-1}}$
  plus an extra $1$ in the $(l+n)^{th}$ spot. (When $n \geq l$ this is
  the sequence used in the definition of $L_{l,l,n}$.) Then
  \[
  Kh(P(-l,l,n)) = q^{-2l-2n-1}t^{-l-n} V[a^{(n)}, E],
  \]
  where the exceptional pair is in the $(l+n)^{th}$ index (where the
  extra $1$ was added). See Figure~\ref{KhPn552} for an example (the
  case $l = 5$ and $n = 2$).

  Note that this formula is consistent with our general formula,
  regardless of whether $n$ is even or odd.
\end{lemma}

\begin{proof}
  We will induct on $n$, starting with $n=0$. By Lemma 2.3 of
  \cite{CK}, $P(-l,l,0) = \overline{T_{2,l}} \# T_{2,l}$ is
  quasi-alternating since $T_{2,l}$ is alternating. Hence, its
  Khovanov homology is contained in two $\delta$-gradings. These must
  be $\delta = \pm 1$ since $P(-l,l,0)$ is slice; in fact, $P(-l,l,n)$
  is slice for all $n$ (see \cite{Starkston} for a nice explanation
  with pictures).

  \begin{figure}
    \begin{center}
      \input{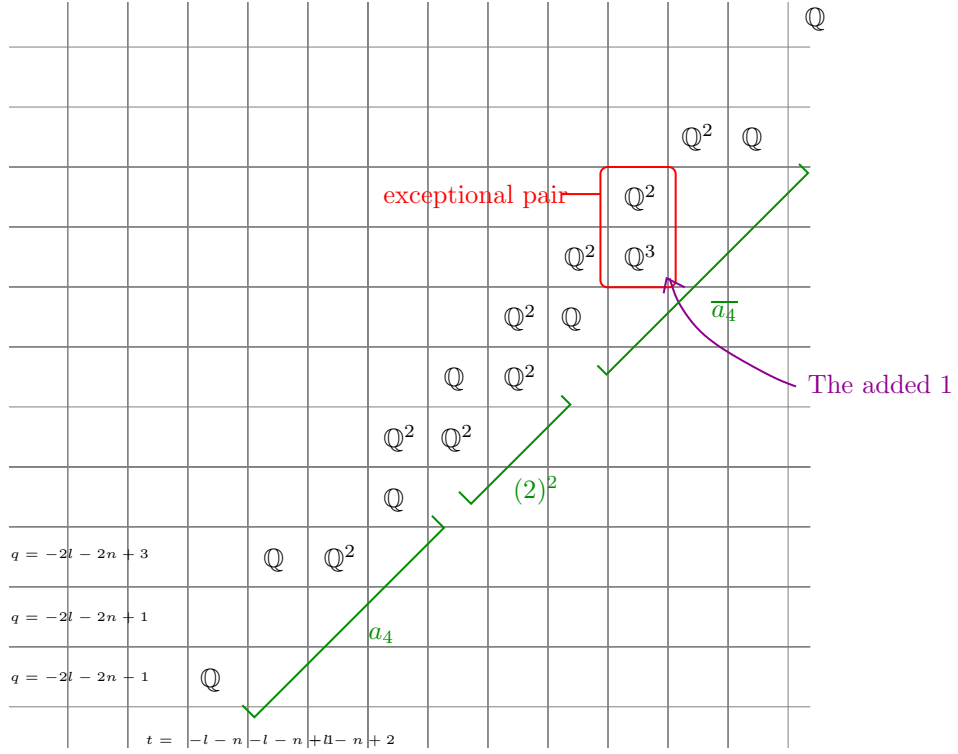}
    \end{center}
    \caption{\label{KhPn552} The Khovanov homology of $P(-5,5,2)$.}
  \end{figure}

  Now, by the thinness just established, the Khovanov polynomial of
  $P(-l,l,0)$ is uniquely determined by the requirement that it be
  well-structured and consistent with its Jones polynomial. To check
  that $q^{-2l-1}t^{-l} V[a^{(0)},E]$ is consistent with the Jones
  polynomial formula given in Lemma~\ref{jones}, plug in $t = -1$. The
  knight's moves (i.e. the sequence \emph{without} the added $1$) give
  us
  \begin{align*}
    -q^{-2l-1}(1 - q^4) \bigg( &\sum_{i=0}^{\frac{l-3}{2}} \bigg(
    (i+1)q^{4i} -iq^{4i+2} \bigg) + \bigg( \frac{l-1}{2} \bigg)
    (q^{2l-2} - q^{2l}) \\
    &+ \sum_{i=0}^{\frac{l-3}{2}} \bigg( \bigg( \frac{l-3}{2} - i
    \bigg) q^{2l+4i+2} - \bigg( \frac{l-1}{2}-i \bigg) q^{2l+4i+4}
    \bigg) \bigg)
  \end{align*}

  which expands to
  \begin{align*}
    &-q^{-2l-1} \bigg( 1 + \sum_{i=0}^{l-4} (-1)^i q^{2i+4} -q^{2l} -
    q^{2l+2} + \sum_{i=0}^{l-4} (-1)^{i+1} q^{2l+2i+6}  + q^{4l+2} \bigg) \\
    &= -q^{-2l-1} + \sum_{i=1}^{l-4} (-1)^{i+1} q^{-2l+2i+4} + (q^{-1}
    + q) + \sum_{i=0}^{l-4} (-1)^i q^{2i+5} - q^{2l+1}.
  \end{align*}

  The exceptional pair adds $q^{-1} + q$, and after reindexing this is
  precisely the formula given by Lemma~\ref{jones}.

  Now suppose the lemma holds for $P(-l,l,n-1)$, and consider
  $P(-l,l,n)$. We will use the skein exact sequence obtained by
  resolving the top crossing on the rightmost strand. This crossing is
  negative, so we must use the second sequence in
  Theorem~\ref{skein}. The standard diagram we use for $P(-l,l,n)$ has
  $l+n$ negative crossings, as listed in
  Table~\ref{GradingShift}. Note that $P(-l,l,n)$ is oriented $RR$ if
  $n$ is odd and $LR$ if $n$ is even, but the corresponding values of
  $n_-$ in Table~\ref{GradingShift} turn out to be the same since $l =
  m$. Similarly, the unoriented resolution, $P(-l,l,n-1)$, has $l+n-1$
  negative crossings. Thus, $\epsilon = -1$. The oriented resolution
  is a diagram for the $2$-component unlink $U_2$, which has Khovanov
  polynomial $q^2 + 2q + q^{-2}$. The sequence is
  \[
  \xymatrix{ \ar[r]^-{f} & q^{-1}Kh(U_2) \ar[r] & Kh(P(-l,l,n)) \ar[r]
    & q^{-2} t^{-1} Kh(P(-l,l,n-1)) \ar[r]^-{f}f &.}
  \]
  For convenience, call the left-hand term $W$, the middle term $X$,
  and the right-hand term $V$. $W$ has two exceptional pairs in $t =
  0$, and $V$ has one exceptional pair in $t = -1$. $X$ has one
  exceptional pair in $t = 0$.

  By induction, plus the grading shifts in the sequence, $V =
  q^{-2l-2n-1}t^{-l-n}V[a^{(n-1)},E]$. The map $f$ preserves the
  $q$-grading and increases the $t$-grading by one.

  Note that $V \oplus W$ looks like the answer we want for $X$, except
  that it has three exceptional pairs rather than one. Two of them
  come from $W$ and have $t = 0$, while the third comes from the
  exceptional pair in $V$ and has $t = -1$. In fact, as discussed
  above, $X$ is a cancellation of $V \oplus W$, and the cancellation
  will cut us down to one exceptional pair. Let $\{e_1, e_2, e_3,
  e_4\}$ be any basis for $W$, in $q$-gradings $1$, $-1$, $-1$, and
  $-3$ respectively (see Figure~\ref{oddlln}). We must determine which
  of the $e_i$ cancel with a generator of $V$. If we could show that
  $e_1$ and $e_2$ survive while $e_3$ and $e_4$ cancel, this would
  imply the formula we want.

  \begin{figure}
    \begin{center}
      \input{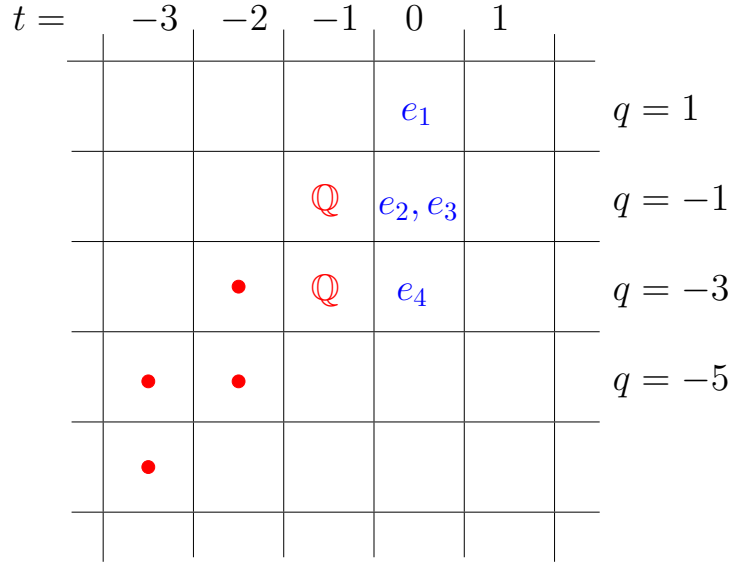}
    \end{center}
    \caption{\label{oddlln} The inductive step of
      Lemma~\ref{oddhalfway}. The red copies of $\Q$ depict the
      exceptional pair of $V$. The red dots are meant to suggest that
      $V$ has generators to the left of the two copies of $\Q$; note
      that it may also have generators to the right of these copies of
      $\Q$ as well, if $n$ is small.}
  \end{figure}

  First, $e_1$ cannot cancel because $V_{1,-1} = 0$ (there is nothing
  ``to the left of $e_1$'' in $V$). Next, note that $\dim V_{-1,-1} =
  \dim V_{-5,-2} + 1$ because of the exceptional pair. But we must
  have $\dim X_{-1,-1} = \dim X_{-5,-2}$ since the only exceptional
  pair of $X$ lies in $t = 0$. So a one-dimensional subspace of
  $\Q(e_2, e_3)$ must cancel (or, since we never pinned down $e_2$ and
  $e_3$, we can just say that $e_3$ cancels).

  Finally, if $e_4$ did not cancel, then it would need to be in an
  exceptional pair in $X$, since we know $\dim X_{1,1} = \dim V_{1,1}
  = \dim V_{5,2} = \dim X_{5,2}$. But since the $q$-grading of $e_4$
  is $-3$, this would imply that the $s$-invariant of $P(-l,l,n)$ is
  $-2$. This contradicts the sliceness of $P(-l,l,n)$ mentioned above,
  since Rasmussen proves in \cite{Rslice} that slice knots must have
  $s = 0$. Hence $e_4$ must cancel, and we have proved our formula for
  $Kh(P(-l,l,n))$.
\end{proof}

\begin{remark} Alternatively, the results of Greene in \cite{Greene}
  imply that $P(-l,l,n)$ is quasi-alternating for $n < l$, so for
  these values of $n$ we could obtain the above result simply by
  looking at the Jones polynomial. However, the general formula for
  the Jones polynomial of pretzel knots (computed by Landvoy in
  \cite{Landvoy}) is a bit complicated, and beginning the induction in
  Lemma~\ref{oddhalfway} at $n=0$ rather than $n=l-1$ makes for a
  cleaner argument.
\end{remark}

\subsection{$P(-l,m,n)$ for odd $l$ and odd $n$.}

\begin{theorem}\label{oddloddnthm}
  The formulas given in Section~\ref{genform} hold for $P(-l,m,n)$
  when $l$ and $n$ are odd.
\end{theorem}

\begin{proof}
  We will use the base case $m = l$ to induct. Assume our formula
  holds for $P(-l,m-1,n)$. We will prove it for $P(-l,m,n)$ using the
  skein exact sequence for the top crossing in the middle strand,
  which is a negative crossing regardless of whether $m$ is even or
  odd.

  First assume $m$ is even. Our diagram for $P(-l,m,n)$ is oriented
  $LL$ and so has $l+m$ negative crossings (see
  Table~\ref{GradingShift}). The unoriented resolution is
  $P(-l,m-1,n)$, and its diagram is oriented $RR$ with $m+n-1$
  negative crossings. Hence $\epsilon = n-l-1$. The oriented
  resolution is a diagram for the (right-handed, positively oriented)
  torus link $T_{n-l,2}$. The sequence is
  \[
  \xymatrix@=18pt{ \ar[r]^-{f} & q^{-1}Kh(T_{n-l,2}) \ar[r] & Kh(P(-l,m,n))
    \ar[r] & q^{3n-3l-2} t^{n-l-1} Kh(P(-l,m-1,n)) \ar[r]^-{f} &.}
  \]
  Again, call the left-hand term $W$, the middle term $X$, and the
  right-hand term $V$. $W$ has exceptional pairs in $t = 0$ and $t =
  n-l$. $V$ has an exceptional pair in $t = n-l-1$. $X$ has one
  exceptional pair in $t = 0$.

  Since $V$ is the Khovanov homology of a pretzel knot (up to a
  shift), we may write $V = L \oplus U$ as in
  Definition~\ref{LUDef}. By induction, and because of the grading
  shifts in the above sequence, $L$ is based at $(q,t) = ((-m-(n-1)) +
  (n-l-1), (-2m-2(n-1)-1) + (3n-3l-2)) =(-l-m, -3l-2m+n-1)
  $. Similarly, $U$ is based at $(l-m+1,l-2m+n-1)$. Note that these
  values match up with the $LL$ row of Table~\ref{GradingShift}, which
  is good since $P(-l,m,n)$ is oriented $LL$.

  To analyze the cancellations, we will fix some notation. As we just
  noted, $W$ has two exceptional pairs, one in $t = 0$ and the other
  in $t = n-l$. Pick generators $\{e_1,e_2\}$ for the first
  exceptional pair and $\{f_1, f_2\}$ for the second, such that $e_1$
  and $f_1$ have the higher $q$-gradings (Figure~\ref{oddloddn} shows
  the case $m = l+1$).

  First consider the case $m = l + 1$, as in Figure~\ref{oddloddn}. No
  generators of $W$ except for the $e_i$ and $f_i$ could possibly
  cancel. If either of the $e_i$'s cancelled, $X$ could not have an
  exceptional pair in $t = 0$, a contradiction. On the other hand,
  both $f_i$'s must cancel for $X$ to be well-structured, since $X$
  has no exceptional pairs except in $t = 0$.

  The result of these cancellations is that, first, $L$ loses its
  exceptional pair. Definition~\ref{lower}\eqref{oddmneql} shows that
  this behavior is precisely what was predicted for the lower summand
  of $P(l,l+1,n)$, in contrast with
  Definition~\ref{lower}\eqref{oddmeql}. We already saw that $L$ was
  based at the correct point. Figure~\ref{oddloddn} shows that we also
  pick up an upper summand from $W'$, where $W'$ denotes $W$ without
  its top exceptional pair. Recall that $V$ did not start with an
  upper summand. The new upper summand is based at $(0,n-l-3)$, in
  accord with row $LL$ of Table~\ref{GradingShift}. It has the correct
  form as specified in Definition~\ref{upper}\eqref{upperoeo}, since
  the formula there in the case $g = 0$ just gives the Khovanov
  homology of $T_{n-l,2}$ (minus the top exceptional pair), up to a
  grading shift.
  
  \begin{figure}
    \begin{center}
      \input{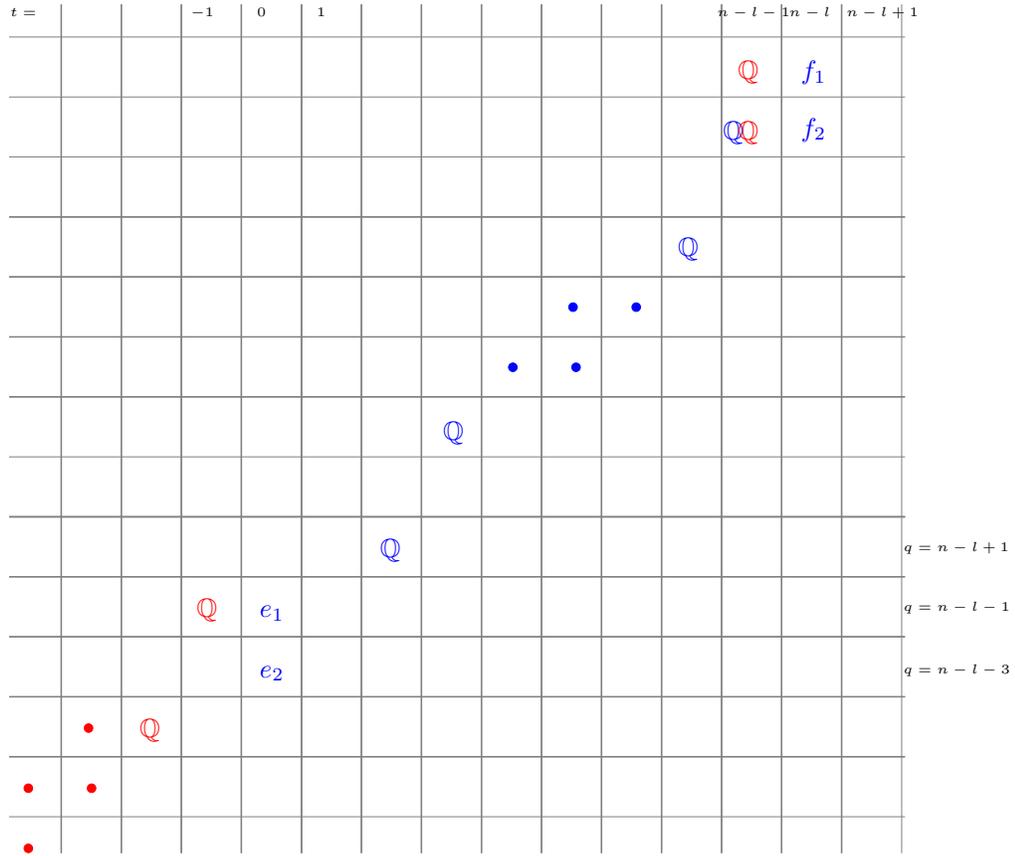}
    \end{center}
    \caption{\label{oddloddn}{The case $m = l+1$ in
        Theorem~\ref{oddloddnthm}. Red denotes $V$ and blue denotes
        $W$.}}
  \end{figure}
 
  Now consider the case $m > l + 1$ ($m$ is still assumed to be
  even). $L$ is based at $t = -l-m$, and it occupies $2l + 1$ columns,
  so the highest $t$-grading of $L$ is $l-m$. But this value is less
  than $-1$, while the lowest $t$-grading in $W$ is $0$. Hence no
  cancellations between $L$ and $W$ can occur, and we have $X = L
  \oplus X'$ where $X'$ is a cancellation of $U \oplus W$. Note that
  both $U$ and $W$ are contained in the same two adjacent
  $\delta$-gradings, namely $\delta = n-l-1$ and $\delta = n-l-3$.

  Now, $Kh(P(-l,m-1,n))$ has zero as its highest $t$-grading (its one
  exceptional pair lies in its highest $t$-grading, so this grading
  must be $0$). Hence $V$ (or equivalently $U$) has $n-l-1$ as its
  highest $t$-grading, and it has an exceptional pair in this
  column. Also, $W$ has one of its two exceptional pairs in $t =
  n-l$. Lemma~\ref{hstandard} applies, and there is a standard
  cancellation between $t = n-l-1$ and $t = n-l$.

  Now Remark~\ref{nomore} applies, and any further cancellations would
  be cancellations of $U' \oplus W'$ (using the notation of
  Remark~\ref{nomore}). But $U'$ no longer has any exceptional pairs,
  and $W'$ only has one of them. The exceptional pair of $X$ must be
  contained in its upper summand $X'$, so it must come from $U' \oplus
  W'$. Thus, by Lemma~\ref{noeps}, no further cancellations are
  possible.

  It is easy now to verify our formula inductively. Recall that $U$
  began its life in the form of
  Definition~\ref{upper}\eqref{upperooo}. The standard cancellation
  just replaces the exceptional pair of $U$ with a knight's move; in
  effect, $U$ becomes $V[a]$ (up to a shift), where $a$ is the
  sequence $a_{m-1-l} \cdot (\frac{m-1-l}{2})^{n-m+1} \cdot
  \overline{a_{m-1-l}}$. Adding in the rest of $W$ amounts to adding
  the sequence $(1,0)^{(n-l)/2 }$, coordinate-wise, to the right side
  of $a$, where the first $1$ carries an exceptional pair. (``To the
  right side'' means the addition is done such that the final $0$ of
  $(1,0)^{(n-l)/2}$ lines up with the last nonzero index of $a$.) This
  addition replaces $\overline{a_{m-l-1}}$ with $\overline{c_{m-l-1}}$
  and $(\frac{m-1-l}{2})^{n-m+1}$ with
  $(\frac{m-l+1}{2},\frac{m-l-1}{2})^{(n-m+1)/2}$. The exceptional
  pair is on the first instance of $\frac{m-l+1}{2}$, giving us the
  upper summand we want as specified in
  Definition~\ref{upper}\eqref{upperoeo}. We noted before that the
  summands are based at the right points, so we have proved our
  formula.

  Now assume $m$ is odd. The diagram for $P(-l,m,n)$ is oriented $RR$
  and has $m+n$ negative crossings, while the diagram for the
  unoriented resolution $P(-l,m-1,n)$ is oriented $LL$ and has $l+m-1$
  negative crossings. Thus $\epsilon = l-n-1$. The oriented resolution
  is $-T_{n-l,2}$, the negatively oriented right-handed $(n-l,2)$
  torus link. The skein sequence is
  \[
  \xymatrix@=18pt{ \ar[r]^-{f} & q^{-1}Kh(-T_{n-l,2}) \ar[r] &
  Kh(P(-l,m,n)) \ar[r] & q^{3l-3n-2} t^{l-n-1} Kh(P(-l,m-1,n))
  \ar[r]^-{f} &.}
  \]

  Again, call the left-hand term $W$, the middle term $X$, and the
  right-hand term $V$. $W$ has exceptional pairs in $t = l-n$ and $t =
  0$, while $V$ has an exceptional pair in $t = l-n-1$. $X$ has an
  exceptional pair in $t = 0$. Write $V = L \oplus U$. By induction,
  and the grading shifts in the sequence, $L$ is based at
  $(-m-n,-2m-2n-1)$ and $U$ is based at $(2l-m-n+1,
  4l-2m-2n-1)$. These values agree with the $RR$ row of
  Table~\ref{GradingShift}.

  The analysis of the cancellations proceeds as above. The lowest
  $t$-grading in $W$ is $t = l-n$, where it has an exceptional pair,
  and $V$ has its exceptional pair in $t =
  l-n-1$. Lemma~\ref{lstandard} ensures a standard cancellation occurs
  between $t = l-n$ and $t = l-n-1$. Afterwards, Remark~\ref{nomore}
  and Lemma~\ref{noeps} apply, and no more cancellations can occur (in
  particular, no cancellation occurs in the highest $t$-grading). Easy
  checks now verify our formulas, as before.
\end{proof}

\subsection{$P(-l,l+1,n)$ for odd $l$ and general $n$.}

Now we will compute $Kh(P(-l,l+1,n))_{LL}$ for even $n$, given that we
know our formula holds for odd $n$. 


\begin{lemma}\label{llplus1lemma}
  For $n \geq l+1$, $Kh(P(-l,l+1,n)_{LL})$ is given by the formula in
  Theorem~\ref{maintheorem}.
\end{lemma}

\begin{proof}
  We have already proved the result for odd $n > l+1$ in
  Theorem~\ref{oddloddnthm}. Hence, we may assume $n$ is even. Resolve
  the top crossing on the rightmost strand, a positive crossing. Our
  diagram for $P(-l,l+1,n)_{LL}$ has $2l+1$ negative crossings, while
  the unoriented resolution (a diagram for the unknot) has $l+n-1$
  negative crossings. Hence $\epsilon = n-l-2$. The oriented
  resolution is a diagram for $P(-l,l+1,n-1)$. The skein sequence is
  \[
  \xymatrix@=18pt{\ar[r]^-{f} & q^{3n-3l-4} t^{n-l-1} Kh(U) \ar[r] &
    Kh(P(-l,l+1,n)) \ar[r] & q Kh(P(-l,l+1,n-1)) \ar[r]^-{f} &,}
  \]
  where $U$ is the unknot, with Khovanov polynomial $q +
  q^{-1}$. Denote the left, middle, and right terms by $W$, $X$, and
  $V$ respectively. $V$ has an exceptional pair in $t = 0$ and $W$ has
  an exceptional pair in $t = n-l-1$. $X$ has exceptional pairs in
  both $t = 0$ and $t = n-l-1$ by Proposition~\ref{linking}, since the
  linking number of the two components of $P(-l,l+1,n)_{LL}$ is
  $\frac{n-l-1}{2}$. 

  Suppose $n > l+1$, and write $V = L \oplus U$. It is clear that $L$
  and $U$ are based at the right points to form the lower and upper
  summands for $Kh(P(-l,l+1,n)_{LL})$, by looking at
  Table~\ref{GradingShift} and the grading shift of $q$ in the above
  sequence. The situation is shown in
  Figure~\ref{llplus1odd}. There is no cancellation since
  $V_{*,n-l-1}=0$, $\dim W_{*,n-l-1} = 2$, and $X$ has an exceptional
  pair in $t = n-l-1$.

  \begin{figure}
    \begin{center}
      \input{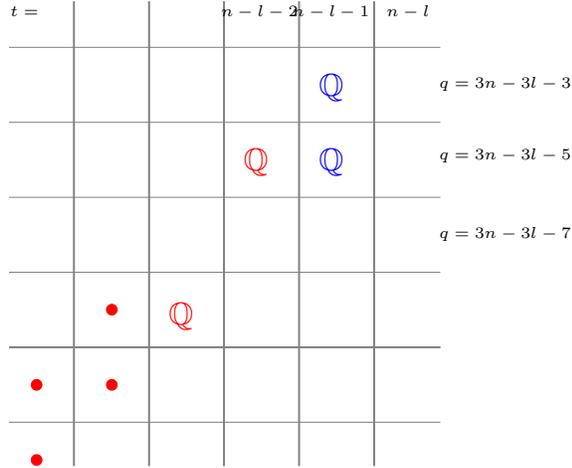}
    \end{center}
    \caption{\label{llplus1odd}{The case $n > l+1$ in
        Lemma~\ref{llplus1lemma}. Red denotes $V$ and blue denotes
        $W$.}}
  \end{figure}

  If $n = l+1$, then $V$ comes from the Khovanov homology of
  $P(-l,l+1,l)$, which was not covered in
  Theorem~\ref{oddloddnthm}. However, $P(-l,l+1,l) = P(-l,l,l+1)$, and
  we computed the Khovanov homology of this knot in
  Lemma~\ref{oddhalfway}. So we still know what $V$ is, by looking at
  the formula above for $P(-l,l,l+1)$.

  In fact, there is still no cancellation in the sequence. We have
  $\dim V_{*,0} = \dim W_{*,0} = 2$, but $X$ needs two exceptional
  pairs in $t = 0$. So $X$ is just $V \oplus W$. One can check that
  this result agrees with our general formula.
\end{proof}

\subsection{$P(-l,m,n)$ for odd $l$ and even $n$.}

\begin{theorem}\label{oddlevennthm}
  When $l$ is odd and $n$ is even, $Kh(P(-l,m,n)_{LR})$ is given by
  the formula in Theorem~\ref{maintheorem}.
\end{theorem}

\begin{proof}
  Our base case is $m = l+1$, where we can appeal to
  Lemma~\ref{llplus1lemma}. Note that for $P(-l,m,n)$ with $l$ odd,
  $m$ even, and $n$ even, switching orientations from $LL$ to $LR$
  picks up a factor of $q^{-6((n-m)/2)} t^{-2((n-m)/2)} = q^{3m-3n}
  t^{m-n}$, since the linking number of the relevant component with
  the rest of the link is $(n-m)/2$. Comparing this shift with the
  data in Table~\ref{GradingShift}, we see that our formulas hold for
  $P(-l,m,n)_{LR}$ as soon as they hold for $P(-l,m,n)_{LL}$. In
  particular, our formulas hold for $P(-l,l+1,n)_{LR}$.

  Now assume $m > l+1$. The diagram for $P(-l,m,n)_{LR}$ has $l+n$
  negative crossings. We will resolve the top crossing on the middle
  strand, a positive crossing. The unoriented resolution, a diagram
  for the right-handed torus knot $T_{n-l,2}$, has $l+m-1$ negative
  crossings. Thus $\epsilon = m-n-1$. The oriented resolution is a
  diagram for $P(-l,m-1,n)_{LR}$, and the skein sequence is
  \[
  \xymatrix@=18pt{ \ar[r]^-{f} & q^{3m-3n-1}t^{m-n}
    Kh(\overline{T_{n-l,2}}) \ar[r] & Kh(P(-l,m,n)) \ar[r] & q
    Kh(P(-l,m-1,n)) \ar[r]^-{f} &.}
  \]
  Call the left-hand term $W$, the middle term $X$, and the right-hand
  term $V$. $W$ has one exceptional pair in $t = m-n$. If $m$ is even,
  $V$ has one exceptional pair in $t = 0$ and $X$ has exceptional
  pairs in $t = 0$ and $t = m-n$. Similarly, if $m$ is odd, $V$ has
  exceptional pairs in $t = 0$ and $t = m-n-1$, and $X$ has one
  exceptional pair in $t = 0$.

  Write $V = L \oplus U$. It is easy to see, by looking at the $LR$
  row of Table~\ref{GradingShift} and at the shift of $q$ in the above
  sequence, that $L$ and $U$ are based at the correct points. The
  exceptional pair, or pairs, of $V$ fall in the $U$ summand.

  We need only consider cancellations of $U \oplus W$. The exceptional
  pair of $W$ is in $t = m-n$. When $m$ is odd, $U$ has two
  exceptional pairs ($t = 0$ and $t = m-n-1$). When $m$ is even, $U$
  has one exceptional pair in $t = 0$. Hence, when $m$ is even, no
  cancellations can occur by Lemma~\ref{noeps} (note that we only need
  to consider the case $m \leq n$), and it is easy to check our
  formula. When $m$ is odd, Lemma~\ref{lstandard} guarantees a
  standard cancellation between $t = m-n-1$ and $t = m-n$, while
  Remark~\ref{nomore} and Lemma~\ref{noeps} preclude any further
  cancellations. Again, one can now check our formula, finishing the
  computation.
\end{proof}

\section{Proof of the general formula for even $l$}
\label{sec:proof-gener-form}

\subsection{$P(-l,l,n)$ for even $l$.}
We now carry out the strategy of Section~\ref{strategy} for even
$l$. We first compute $Kh(P(-l,l,n))$; recall that we are using the
$RL$ orientation throughout Section~\ref{sec:proof-gener-form}. This
computation will not be used as a base case for induction, but we need
to deal with it anyway since it is not covered by our other
computations.

\begin{theorem}\label{evenhalfway0}
  Let $l \geq 2$ be even and let $n \geq 0$. Let $c$ denote the
  sequence $(1) \cdot c_{l-4} \cdot (\frac{l-2}{2})^6 \cdot
  \overline{c_{l-4}} \cdot (0,1)$. For even $n \geq 0$, let $d^{(n)}$
  denote the sequence $(1,-1,\ldots,1,-1,2)$ of length $n+1$ (when $n
  = 0$, $d^{(0)}$ is just the sequence $(2)$). For odd $n \geq 1$, let
  $d^{(n)}$ be the sequence $(1,-1,\ldots,1)$ of length $n$. Define
  \[
  c^{(n)} = c + d^{(n)},
  \]
  where the addition to $c$ starts in the $(l+1)^{st}$ spot.

  Then if $n < l$,
  \[
  Kh(P(-l,l,n)_{RL}) = q^{l+n-1} V[c^{(n)}, E],
  \]
  where there are exceptional pairs in the first index and the last
  index, plus two exceptional pairs in the $(l+n+1)^{st}$ index if $n$
  is even.

  If $n \geq l$, then $Kh(P(-l,l,n)_{RL})$ is as described in
  Theorem~\ref{maintheorem}.
\end{theorem}

\begin{proof}
  The proof is by induction on $n$, as before. When $n = 0$, we know
  $P(-l,l,0)_{RL}$ is quasi-alternating, and its Jones polynomial was
  given in Equation~\eqref{RLjones}. A check similar to that in
  Lemma~\ref{oddhalfway} verifies the formula.

  Suppose the lemma holds for $n-1$, and consider $P(-l,l,n)$. Again,
  we will use the skein sequence from resolving the top crossing on
  the rightmost strand, a positive crossing. Our diagram for
  $P(-l,l,n)$ has no negative crossings. The unoriented resolution, a
  diagram for the 2-component unlink, has $l+n-1$ negative crossings
  (regardless of the orientation chosen). Hence $\epsilon = l +
  n-1$. The oriented resolution is a diagram for
  $P(-l,l,n-1)_{RL}$. The skein sequence for a positive crossing is
  \[
  \xymatrix{\ar[r]^-{f} & q^{3l+3n-1} t^{l+n} Kh(U_2) \ar[r] &
    Kh(P(-l,l,n)) \ar[r] & q Kh(P(-l,l,n-1)) \ar[r]^-{f} &.}
  \]
  Denote the left, middle, and right terms by $W$, $X$, and $V$
  respectively. $W$ has two exceptional pairs in $t = l+n$. If $n$ is
  odd, $V$ has four exceptional pairs (one in $t = 0$, one in $t =
  2l$, and two in $t = l+n-1$), and $X$ has two exceptional pairs (one
  in $t = 0$ and one in $t = 2l$). On the other hand, if $n$ is even,
  then $V$ has two exceptional pairs (in $t = 0$ and $t = 2l$), and
  $X$ has four exceptional pairs (one in $t = 0$, one in $t = 2l$, and
  two in $t = l+n$). The map $f$ preserves the $q$-grading and
  increases the $t$-grading by one.

  First, suppose $n$ is odd and $n < l$; see the left side of
  Figure~\ref{evenlln} for reference. Note that $d^{(n)}$ is just
  $d^{(n-1)}$ with the terminating $2$ replaced by a $1$, and thus
  $c^{(n)}$ is related to $c^{(n-1)}$ in a similar way.

  Hence $V \oplus W$ looks like the answer we want for $X$, except
  that it has six exceptional pairs rather than two. In fact, the
  cancellation process will cut us down to two exceptional pairs. Let
  $\{e_1, e_2, e_3, e_4\}$ be any basis for $W$, in $q$-gradings
  $3l+3n+1$, $3l+3n-1$, $3l+3n-1$, and $3l+3n-3$ respectively. We must
  determine which of the $e_i$ cancel with a generator of $V$. If we
  could show that $e_1$ survives while the rest cancel, this would
  imply the formula we want (by a simple check).

  \begin{figure}
    \begin{center}
      \input{evenlln.pstex_t}
    \end{center}
    \caption{\label{evenlln}{The case of odd $n$ in
        Theorem~\ref{evenhalfway0}. The left side depicts the case $n
        < l$. The middle depicts $n = l+1$. The right side depicts the
        case $n > l+1$. As usual, red denotes $V$ and blue denotes
        $W$.}}
  \end{figure}

  First, $e_1$ cannot cancel because there is nothing ``to the left of
  $e_1$'' in $V$. Next, note that $\dim V_{l+n-1,3l+3n-1} = \dim
  V_{l+n-2,3l+3n-5} + 2$ because of the exceptional pairs in $V$. But
  we must have $\dim X_{l+n-1,3l+3n-1} = \dim X_{l+n-2,3l+3n-5}$ since
  $X$ has no exceptional pair in $t = l+n-1$ or $t = l+n-2$. So both
  $e_2$ and $e_3$ must cancel.

  Finally, if $e_4$ did not cancel, then we would have $\dim
  X_{l+n-1,3l+3n-3} = \dim X_{l+n,3l+3n+1} + 1$. But we need $\dim
  X_{l+n-1,3l+3n-3} = \dim X_{l+n,3l+3n+1}$, again because the only
  exceptional pairs of $X$ are in $t = 0$ and $t=2l$. Hence $e_4$ must
  cancel, and we have proved our formula when $n$ is odd and less than
  $l$.

  The next case for odd $n$ is $n = l+1$. The middle of
  Figure~\ref{evenlln} is a reference here. A similar argument implies
  that $e_2$, $e_3$, and $e_4$ all cancel, and a quick check verifies
  that our results for the lower and upper summands of $X$ agree with
  Definition~\ref{lower}\eqref{evenmeqlodd} and
  Definition~\ref{upper}\eqref{uppereeo}.

  If $n$ is odd and $n > l+1$, then the right side of
  Figure~\ref{evenlln} depicts the situation. The generator $e_1$
  still cannot cancel. But logic very similar to before implies that
  $e_4$ and a one-dimensional subspace of $(e_2,e_3)$ must cancel,
  since $X$ has no exceptional pairs in $t = l+n$. Thus, the lower
  summand loses an exceptional pair (as predicted by
  Definition~\ref{lower}), and the upper summand turns an exceptional
  pair into a knight's move.

  Finally, suppose $n$ is even; luckily, this case is easier. First,
  if $n < l$, we want to show that only $e_4$ cancels (using the above
  notation). Again, $e_1$ cannot cancel. If some combination of $e_2$
  and $e_3$ cancelled, then we would have $\dim X_{l+n-1,3l+3n-1} <
  \dim X_{l+n-2,3l+3n-5}$, an impossibility since $X$ has no
  exceptional pairs in $t = l+n-2$. So both $e_2$ and $e_3$ survive.

  Since $n < l$, $P(-l,l,n)$ is quasi-alternating by \cite{Greene}, so
  $X_{l+n,3l+3n-3} = 0$ for $\delta$-grading reasons, and $e_4$ has
  nowhere to live. The cancellation is responsible for the next $(-1)$
  in the sequence $d^{(n)}$, and $e_2$ and $e_3$ are responsible for
  the $(2)$ following it.

  On the other hand, if $n > l$, then $X$ has two exceptional pairs in
  $t = l+n$ but $V_{l+n} = 0$, so none of the $e_i$ can cancel. If $n
  = l$, the same argument holds: $X$ must have three exceptional pairs
  in $t = l+n$ but $\dim V_{l+n}$ is only two, so none of the $e_i$
  can cancel. Thus $e_1$ and (say) $e_2$ add an exceptional pair to
  the lower summand, and $e_3$ and $e_4$ add an exceptional pair to
  the upper summand. These additions are precisely what we were
  expecting, completing the inductive argument.
\end{proof}

\subsection{$P(-l,l+1,n)$ for even $l$.}

Now we will make a similar computation for $P(-l,l+1,n)$ which we will
use as the base case in the induction to follow. As with odd $l$, we
do not need to start with $n=0$. We can use $n=l$ instead; since
$P(-l,l+1,l) = P(-l,l,l+1)$, the previous section tells us the formula
for $Kh(P(-l,l+1,l))$.

\begin{lemma}\label{evenhalfway} 
  For $n \geq l+1$, $Kh(P(-l,l+1,n))$ is given by the formula in
  Theorem~\ref{maintheorem}.
\end{lemma}

\begin{proof} Resolve the top crossing on the rightmost strand, a
  positive crossing. Our diagram for $P(-l,l+1,n)$ has no negative
  crossings, while the unoriented resolution (a diagram for the
  unknot) has $l+n-1$ negative crossings. The oriented resolution is a
  diagram for $P(-l,l+1,n-1)$. Hence the skein sequence is
  \[
  \xymatrix{ \ar[r]^-{f} & q^{3l+3n-1} t^{l+n} Kh(U) \ar[r] &
    Kh(P(-l,l+1,n)) \ar[r] & q Kh(P(-l,l+1,n-1)) \ar[r]^-{f} &,}
  \]
  where $U$ is the unknot, with Khovanov polynomial $q +
  q^{-1}$. Denote the left, middle, and right terms by $W$, $X$, and
  $V$ respectively. $W$ has an exceptional pair in $t = l+n$ (in fact,
  $W$ consists of this exceptional pair). If $n$ is even, $V$ has one
  exceptional pair in $t = 0$, and $X$ has two exceptional pairs ($t =
  0$ and $t = l+n$) since the linking number of the two components of
  $P(-l,l+1,n)$ is $\frac{l+n}{2}$. If $n$ is odd, $V$ has two
  exceptional pairs ($t = 0$ and $t = l+n-1$), and $X$ has one
  exceptional pair in $t = 0$. It is clear that the upper and lower
  summands of $V$ are based at the correct points to become the upper
  and lower summands we want for $X$, after some cancellations.

  Figure~\ref{llplus1even} depicts some generators of $V$ and $W$. For
  $t \geq l+n-2$, $V$ has only the generators shown. One can check
  this statement either inductively, if $n > l+1$, or by looking at
  the formula for $P(-l,l,l+1) = P(-l,l+1,l)$ if $n = l+1$.

  If $n$ is even then there is no cancellation since $V_{l+n,*}=0$,
  $\dim W_{l+n,*} = 2$, and $X$ has an exceptional pair in $t =
  l+n$. So $X = V \oplus W$, which agrees with our
  formulas.

  \begin{figure}
    \begin{center}
      \input{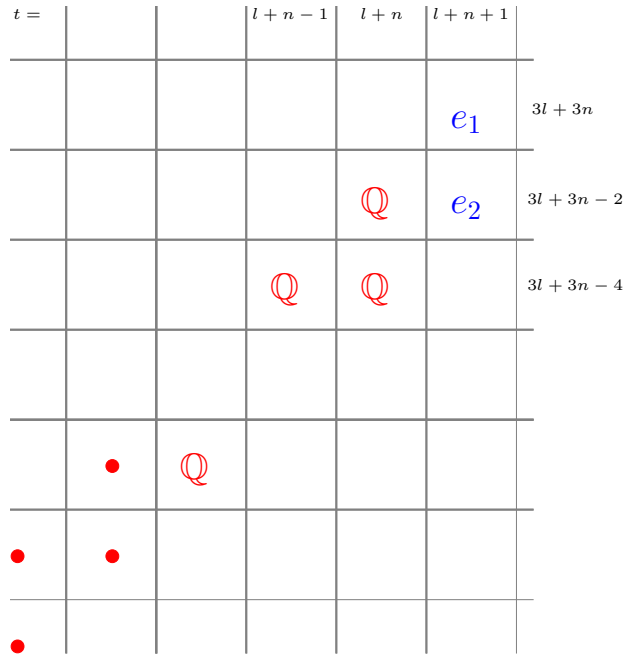}
    \end{center}
    \caption{\label{llplus1even}{The case of odd $n$ in
        Lemma~\ref{evenhalfway}. In $t \geq l+n-2$, $V$ is zero except
        for the copies of $\Q$ shown. Red denotes $V$ and blue denotes
        $W$.}}
  \end{figure}

  If $n$ is odd, pick a basis $\{e_1,e_2\}$ for $W$, in $q$-gradings
  $3l+3n$ and $3l+3n-2$ respectively. Now, $V_{l+n-1,3l+3n} = 0$; this
  follows from our formulas but can be seen most easily in
  Figure~\ref{llplus1even}. Hence $e_1$ cannot cancel. On the other
  hand, $e_2$ must cancel: otherwise it would be in a knight's move,
  but $V_{l+n-1,3l+3n-6} = 0$ and $V_{l+n+1,3l+3n+2} = 0$. Thus our
  formula is verified.
\end{proof}

\subsection{$P(-l,m,n)$ for even $l$.}

\begin{theorem}\label{finalthm}
  When $l$ is even, $Kh(P(-l,m,n)_{RL})$ is given by
  the formula in Theorem~\ref{maintheorem}.
\end{theorem}

\begin{proof}
  We already did the case of $m = l$, so we can induct on $m \geq
  l+1$. The base case $m = l+1$ was also done above. Assume that our
  formula holds for $P(-l,m-1,n)$; we will prove it for $P(-l,m,n)$
  using the skein exact sequence for the top crossing in the middle
  strand (a positive crossing). The diagram for $P(-l,m,n)$ has no
  negative crossings. The unoriented resolution is a diagram for the
  right-handed torus link $T_{n-l,2}$. If $n$ is odd, the unoriented
  resolution diagram has $l+m-1$ negative crossings, and $T_{n-l,2}$
  is a knot. If $n$ is even, $T_{n-l,2}$ is a $2$-component link;
  orient it positively, so that the unoriented resolution diagram has
  $l+m-1$ negative crossings. The sequence is
  \[
  \xymatrix{ \ar[r]^-{f} & q^{3l+3m-1} t^{l+m} Kh(T_{n-l,2}) \ar[r] &
    Kh(P(-l,m,n)) \ar[r] & q Kh(P(-l,m-1,n)) \ar[r]^-{f} &,}
  \]
  where $T_{n-l,2}$ is positively oriented when $n$ is even. Call the
  three terms $W$, $X$, and $V$ as usual, and write $V = L \oplus
  U$. As before, $L$ and $U$ are based at the correct points to become
  the upper and lower summands in our desired formula for $X$, after
  cancellations.

  In fact, if $m$ is even, there are no cancellations. To show this
  fact, first suppose $m > l+2$. Nothing in $W$ can cancel with
  anything in $L$ for $t$-grading reasons, so $X = L \oplus X'$ where
  $X'$ is a cancellation of $U \oplus W$.

  Now, $U$ has no exceptional pairs if $n$ is odd, while if $n$ is
  even, $U$ has an exceptional pair in $l+n$. If $n$ is odd, $W$ has
  one exceptional pair in $t = l+m$, and if $n$ is even it has an
  additional one in $t = m+n$. In either case, however, $X'$ must have
  an exceptional pair everywhere that $U$ or $W$ does. Hence $E_{X'}
  \geq E_U + E_W$, so by Lemma~\ref{noeps}, $X' = U \oplus W$. We are
  now done, since adding $W$ to the upper summand $U$ of $V$ produces
  the upper summand we want for $X$.

  The remaining case for even $m$ is $m = l+2$. The same argument
  applies as soon as we can show that nothing in $W$ cancels with
  anything in $L$. But any such cancellation would need to occur
  between $t = 2l+1$ and $t = 2l+2$. In the column $t = 2l+1$, $L$ has
  a single generator, which is part of a knight's move. It cannot
  cancel because, if it did, its knight's move partner could not be
  part of any knight's move or exceptional pair in $X$. So we are done
  with even $m$.

  If $m$ is odd, a few cancellations occur. We may again consider
  cancellations $X'$ of $U \oplus W$. If $n$ is odd, $W$ has one
  exceptional pair in its lowest $t$-grading, $t = l+m$. $V$ has an
  exceptional pair in $t = l+m-1$ and $X$ does not, so by
  Lemma~\ref{lstandard}, there must be a standard cancellation between
  $t$-gradings $l+m-1$ and $l+m$.

  \begin{figure}
    \begin{center}
      \input{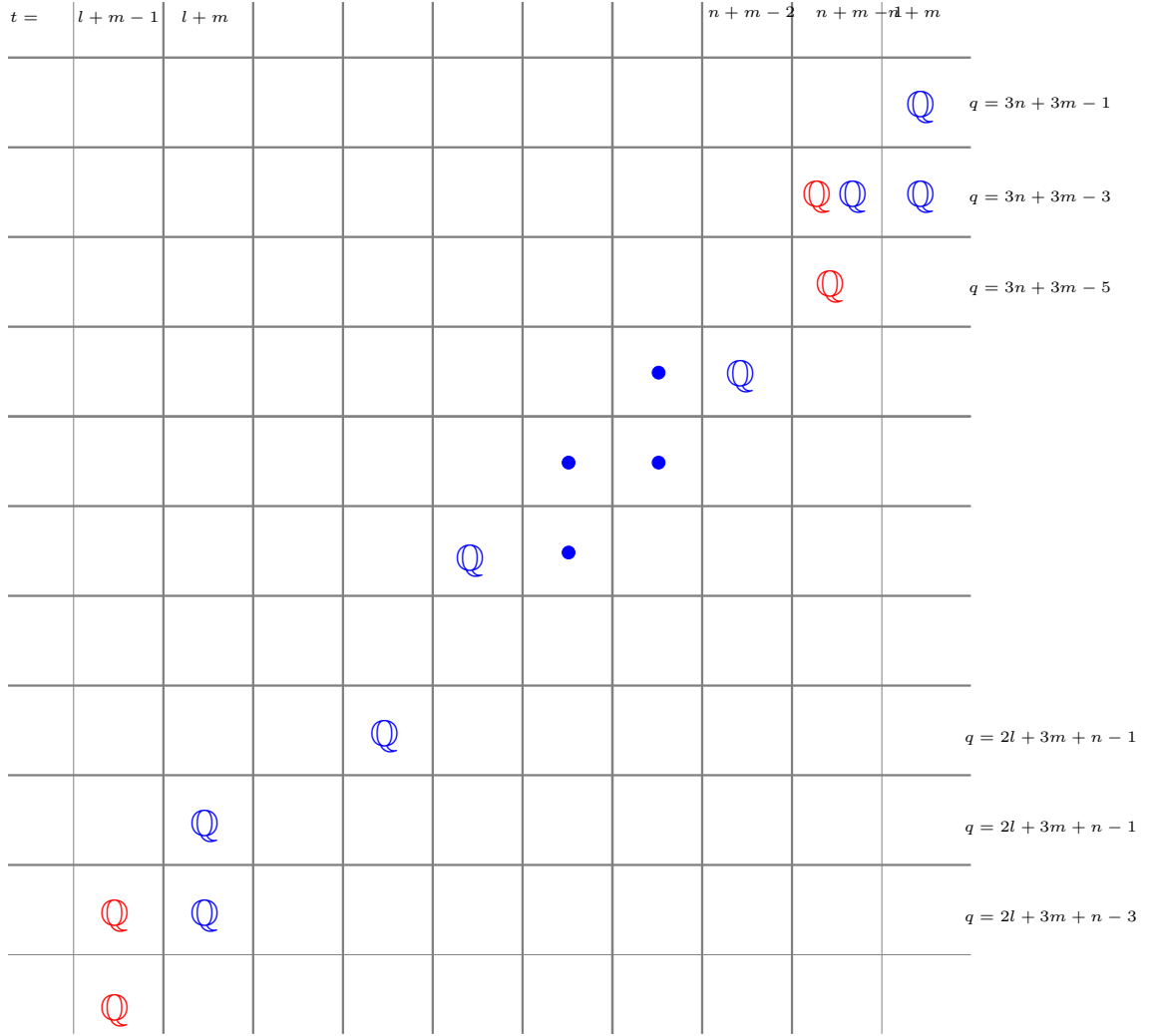}
    \end{center}
    \caption{\label{evenlmn}{The case of odd $n$ in
        Theorem~\ref{finalthm}. Red represents $V$ and blue represents
        $W$.}}
  \end{figure}

  Now, any further cancellations would be cancellations of $U' \oplus
  W'$ as in Remark~\ref{nomore}, but $U'$ and $W'$ have no exceptional
  pairs and are contained in $\delta = n+m-1$ and $\delta = n+m-3$, so
  Lemma~\ref{noeps} precludes any cancellations.

  If $n$ is even (but $m$ is still odd), the same analysis applies to
  the lower exceptional pair of $W$. Now, however, $W$ has an
  additional exceptional pair in $t = (n-l) + (l+m) = n+m$. Also, $X$
  has no exceptional pair in this $t$-grading, and $V$ has an
  exceptional pair in $t = n+m-1$. Hence a standard cancellation must
  occur between $t=n+m-1$ and $t = n+m$ by Lemma~\ref{hstandard}.

  As before, any further cancellations would be cancellations of $U'
  \oplus W'$ as in Remark~\ref{nomore}, but $U'$ and $W'$ have no
  exceptional pairs and are contained in $\delta = n+m-1$ and $\delta
  = n+m-3$, so Lemma~\ref{noeps} precludes any cancellations. We can
  now easily check that we have completed the proof of our general
  formula.
\end{proof}

\bibliographystyle{plain} \bibliography{biblio}

\end{document}